%% file: Con_SLBLR.tex
\documentclass[12pt]{article}
\usepackage[a4paper, left=1.25in, right=1.25in, top=1.2in, bottom=1.2in]{geometry}

\usepackage{graphicx}
\usepackage{makecell}
\usepackage{longtable}
\usepackage{subcaption}
\usepackage{booktabs}
\usepackage{cite}

\usepackage{amsmath}
\usepackage{amssymb}
\usepackage{amsthm}

\usepackage{xcolor}
\usepackage{multirow}

\newtheorem{theorem}{Theorem}[section]    
\newtheorem{lemma}[theorem]{Lemma}        
\newtheorem{proposition}[theorem]{Proposition} 
\newtheorem{assumption}{Assumption}[section]
\newtheorem{condition}{Condition}[section]


\usepackage{float}
\usepackage{graphicx}
\usepackage[colorlinks=true,
            linkcolor=blue,
            citecolor=blue,
            urlcolor=blue]{hyperref}
\begin{document}
\title{Accelerating Level-Value Adjustment for the Polyak Stepsize}

\begingroup
\small

\author{
Anbang Liu\thanks{Center for Intelligent and Networked Systems (CFINS), Tsinghua University, Beijing 100084, China. Email: liuab19@mails.tsinghua.edu.cn}
\and
Mikhail A. Bragin\thanks{University of Connecticut, Storrs, CT 06269, USA. Email: mikhail.bragin@uconn.edu}
\and
Xi Chen\thanks{Center for Intelligent and Networked Systems (CFINS), Tsinghua University, Beijing 100084, China. Email: bjchenxi@mail.tsinghua.edu.cn}
\and
Xiaohong Guan\thanks{Center for Intelligent and Networked Systems (CFINS), Tsinghua University, Beijing 100084, China; MOEKLINNS Laboratory, Xi'an Jiaotong University, Xi'an, Shaanxi, China. Email: xhguan@xjtu.edu.cn}
}

\date{2025.6}

\endgroup

\maketitle

\begin{abstract} 

The Polyak stepsize has been widely used in subgradient methods for non-smooth convex optimization. However, calculating the stepsize requires the optimal value, which is generally unknown. Therefore, dynamic estimations of the optimal value are usually needed. In this paper, to guarantee convergence, a series of level values is constructed to estimate the optimal value successively. This is achieved by developing a decision-guided procedure that involves solving a novel, easy-to-solve linear constraint satisfaction problem referred to as the ``Polyak Stepsize Violation Detector'' (PSVD). Once a violation is detected, the level value is recalculated. We rigorously establish the convergence for both the level values and the objective function values. Furthermore, with our level adjustment approach, calculating an approximate subgradient in each iteration is sufficient for convergence. A series of empirical tests of convex optimization problems with diverse characteristics demonstrates the practical advantages of our approach over existing methods.

\noindent \textbf{Keywords: convex optimization; non-smooth optimization; polyak stepsize; subgradient method; approximate subgradient method}
\end{abstract}

\section{Introduction}

This paper focuses on subgradient methods for the following problem: 
\begin{equation}
\begin{aligned}
& \underset{x\in \mathbb{R}^n}{\min}
& & f(x),
& \text{s.t.}
& & x \in \mathcal{X},
\label{convex optimization problem}
\end{aligned}
\end{equation}
where $f: \mathbb{R}^n\rightarrow\mathbb{R}$ is a convex, non-smooth objective function, and $\mathcal{X}$ is a closed, nonempty, and convex subset of $\mathbb{R}^n$. The optimal value is denoted as $f^{\star} := \inf _{x \in \mathcal{X}} f(x)$ and is assumed to be bounded below. The set of optimal solutions is denoted as $\mathcal{X}^{\star}:=\left\{x \in \mathcal{X}:f(x)=f^{\star} \right\}$ and is assumed to be non-empty. Problem \eqref{convex optimization problem} appears in several contexts such as Lagrangian Relaxation, where non-smooth convex dual functions are optimized \cite{zhao1999surrogate,nedic2001incremental,bragin2015convergence, bragin2022surrogate}, approximation and fitting \cite{osborne2000lasso,boyd2004convex}, and machine learning  \cite{pisner2020support,lim2012consistency}. The objective function of \eqref{convex optimization problem} is non-smooth; thus, gradients may not exist at certain points. To optimize \eqref{convex optimization problem}, subgradient methods have been used, which involve taking a series of steps along subgradient directions as:
\begin{equation}
\label{update equation}
    x^{k+1}=P_{\mathcal{X}}\left(x^k-s^k \cdot g^k\right),
\end{equation}
where $P_{\mathcal{X}}$ is the projection onto $\mathcal{X}$, $s^{k}$ is a positive stepsize and $g^{k}$ is the subgradient of $f$ at $x^{k}$ that satisfies
\begin{equation} 
\label{subgradient inequality}
    f(x) \geq f(x^k) + \big(g^k\big)^{\top} (x - x^k), \forall x \in \mathcal{X}.
\end{equation}
In this paper, we assume that the subgradients are bounded:
\begin{assumption} 
\label{Assumption_1}
\noindent The subgradients of $f$ are bounded, i.e., there exists a scalar $C$, such that
\begin{equation}
\label{subgradient boundness assumption}
    \|g^k\|\leq C, \quad \forall k=0, 1, \dots.
\end{equation}
\end{assumption} 
In \eqref{subgradient boundness assumption} and thereafter in the paper, $\|\cdot\|$ denotes the standard Euclidean norm. Assumption \ref{Assumption_1} is satisfied if, e.g., $f$ satisfies the Lipschitz condition \cite{nesterov2018lectures}. 

The stepsize rules used in subgradient methods are significantly different from those in gradient methods, and have a substantial impact on convergence rates. The adaptive stepsize introduced by Polyak in \cite{polyak1969minimization} has been widely used \cite{goffin1999convergence, nedic2001incremental,orvieto2022dynamics}. The key idea behind it is as follows. From the non-expansive property of the projection $P_{\mathcal{X}}$ and inequality \eqref{subgradient inequality}, a classical result in subgradient methods can be established: 
\begin{equation}
\label{fundamental result}
   \left\|x^{k+1}-x^{\star}\right\|^2-\left\|x^k-x^{\star}\right\|^2 \leq -2 s^k\left(f\left(x^k\right)-f^{\star}\right)+\left(s^k\right)^2\left\|g^k\right\|^2, 
\end{equation}
with $x^{\star}\in\mathcal{X}^{\star}$. With the Polyak stepsize 
\begin{equation}\label{polyak stepsize with optimal obj}
    s^k = \gamma \cdot \frac{f(x^k) - f^{\star}}{\big\|g^k\big\|^2}, \quad 0 < \gamma < \bar{\gamma}<2,
\end{equation} 
the right-hand side of \eqref{fundamental result} is negative, guaranteeing that the sequence $\{x^k\}$ strictly approaches some optimal solution.  Moreover, under additional requirements on ``sharpness” of $f$ and Lipschitz conditions, a linear convergence rate can be established as $\|x^{k+1}-x^{\star}\| \leq r \cdot \|x^{k}-x^{\star}\|$ with $0<r < 1$ (see \cite{polyak1969minimization, polyak1978subgradient} for details).

A primary challenge in employing the Polyak stepsize \eqref{polyak stepsize with optimal obj} is that, in most cases, the optimal value $f^{\star}$ is unknown. To overcome this difficulty, $f^{\star}$ is typically replaced with a level value $\bar{f}_k$ (an estimate of $f^{\star}$) as:
\begin{equation}
    s^k=\gamma \cdot \frac{f\left(x^k\right)-\bar{f}_k}{\left\|g^k\right\|^2}, \quad 0 <  \gamma < \bar{\gamma}<2.
    \label{polyak stepsize with estimation}
\end{equation}
However, the estimation inaccuracy of $\bar{f}_k$ can introduce two critical issues: a) when $\bar{f}_k$ is set ``too high,'' the stepsize becomes excessively small, slowing convergence and potentially leading to premature termination; b) conversely, if $\bar{f}_k$ is set ``too low,'' the stepsize remains excessively large, which can cause oscillations or even divergence. To address these issues, timely and adaptive adjustments of $\bar{f}_k$ are required. 

In the following, two level adjustment approaches, the path-based level adjustment \cite{goffin1999convergence} and the decision-guided level adjustment \cite{bragin2022surrogate}, are reviewed.  In the path-based level adjustment approach \cite{goffin1999convergence}, the level value $\bar{f}_k$ is offset by the best objective function value obtained up to iteration $k$ with a positive parameter $\delta_k$. An oscillation is detected when the objective function value fails to decrease sufficiently and the solution path exceeds a predefined threshold $B$; in such case, $\delta_k$ is reduced. However, the selection of hyperparameters such as $B$ (path length threshold) and $\delta_0$ (initial offsetting hyperparameter) is problem-specific and may affect performance in a major way. 

Originally developed in \cite{bragin2022surrogate}, the decision-guided level adjustment approach does not need hyperparameter tuning. The approach sets a level value that underestimates $f^{\star}$. Once a solution divergence detection mechanism is triggered, the level value is recalculated, which ensures a tighter level value. The specifics of the method are provided ahead.  
Suppose $x^{k}$ is updated per \eqref{update equation} with $s^{k}$ calculated based on \eqref{polyak stepsize with estimation} with a level value of $\bar{f}_{k}$. If the level value is ``appropriate," such that $s^{k}$ satisfies the Polyak formula \eqref{polyak stepsize with optimal obj}, then $x^{k+1}$ gets closer to $x^{\star} \in \mathcal{X}^{\star}$ as  
\begin{equation}
    \left\|x^{\star}-x^{k+1}\right\|^2 \leq \left\|x^{\star}-x^{k}\right\|^2.
    \label{sdd temp}
\end{equation}
Conversely, if \eqref{sdd temp} is violated by $x^{k}$ and $x^{k+1}$, then the stepsize is ``too large'' such that the Polyak formula \eqref{polyak stepsize with optimal obj} is violated, and an adjustment of the level value is required. To operationalize this idea, the 
following constraint satisfaction problem is formulated and serves to trigger the level adjustment:  
\begin{equation}
    \left\{\begin{array}{c}\begin{aligned}
    \left\|x-x^{k(j)+1}\right\|^2 & \leq\left\|x-x^{k(j)}\right\|^2, \\
    \left\|x-x^{k(j)+2}\right\|^2 & \leq\left\|x-x^{k(j)+1}\right\|^2, \\
    & \vdots \\
    \left\|x-x^{k(j)+\eta}\right\|^2 & \leq\left\|x-x^{k(j)+\eta-1}\right\|^2,\\
    \end{aligned}\end{array}\right. \label{divergencedetection}
\end{equation} 
where $x \in \mathbb{R}^n$ is the decision vector, and $\{x^k\}_{k=k(j)}^{k(j)+\eta}$ is the sequence of solutions obtained after the $j$-th level adjustment, which occurs at iteration $k(j)$. 
In \cite{bragin2022surrogate}, the solutions represent the multipliers in Lagrangian dual problems, and the problem is termed a Multiplier-Divergence Detector (MDD) problem. More generally, this is referred to as the ``Solution-Divergence Detector" (SDD) problem throughout the remainder of this paper. The infeasibility of problem~\eqref{divergencedetection} indicates that $\{x^k\}_{k = k(j)}^{k(j) + \eta}$ is not approaching any solution, including $x^{\star}$, and therefore triggers a level adjustment (see~\cite[Eq.~(20)]{bragin2022surrogate} for details). If the SDD problem \eqref{divergencedetection} is feasible, a new inequality $\left\|x-x^{k(j)+\eta+1}\right\|^2 \leq \left\|x-x^{k(j)+\eta}\right\|^2$ is appended to \eqref{divergencedetection}, and feasibility is checked again in the next iteration. 

The primary advantage of the decision-guided procedure lies in its problem-agnostic design and the elimination of hyperparameter tuning. This is particularly valuable in applications where hyperparameter fine-tuning, often reliant on domain knowledge, is time-consuming. By reducing the reliance on such adjustments, the approach becomes more versatile and user-friendly, especially for time-sensitive or large-scale problems (e.g., Lagrangian duals of practical-size discrete programs) where fine-tuning is either impractical or may lead to suboptimal performance.

However, there are three major difficulties. First, while the infeasibility of \eqref{divergencedetection} indicates that $\{x^k\}_{k=k(j)}^{k(j)+\eta}$ does not approach $x^{\star}$, the feasibility of \eqref{divergencedetection} cannot guarantee that $\{x^k\}_{k=k(j)}^{k(j)+\eta}$ gets closer to $x^{\star}$. Consequently, when a level value has not been appropriate at an iteration, the algorithm may still require multiple additional iterations before adjusting the level value. Second, the convergence of the approach was not established in \cite{bragin2022surrogate}. Third, the SDD problem \eqref{divergencedetection} involves quadratic terms, and thus cannot be equivalently transformed to a linear one when the problem solved includes constraints. The feasibility check may therefore be computationally expensive with limited scalability.

In this paper, inspired by \cite{bragin2022surrogate}, we develop a novel decision-guided level adjustment approach, which overcomes the major difficulties outlined above. Our key contributions are as follows. First, we introduce the ``Polyak Stepsize Violation Detector" (PSVD) problem, which triggers level adjustment. The PSVD problem is shown to be tighter than the SDD problem \eqref{divergencedetection}, and thus the level adjustment of \cite{bragin2022surrogate} is accelerated. Second, we rigorously prove the convergence of subgradient methods when using the Polyak stepsize and the PSVD-based level adjustment approach. Third, the PSVD problem is a linear constraint satisfaction problem, enabling efficient feasibility checks. Fourth, we extend the results to approximate subgradient methods. Specifically, we show that, with our PSVD-based level adjustment approach, it is unnecessary to compute an exact subgradient at every iteration—an approximate subgradient suffices to ensure convergence. 

This paper is organized as follows. In Section \ref{section 2},  we propose the PSVD-based level adjustment approach for subgradient methods. In Section \ref{section 3}, we propose the level adjustment approach for approximate subgradient methods. In Section \ref{section 4}, we empirically demonstrate the advantages of our approach on a variety of convex problems. In Section \ref{section 5}, we give the conclusion, summarize the limitations of the method, and discuss the directions for future research.

\section{Polyak Stepsize with Accelerated Decision\\-Guided Level Adjustment for Subgradient \\Methods} \label{section 2} 

In this section, we develop the Polyak Stepsize with Accelerated Decision-Guided Level Adjustment (PSADLA for short) for subgradient methods. Rather than detecting solution divergence as in \cite{bragin2022surrogate}, our level-adjustment approach directly detects Polyak stepsize violations directly using a novel ``Polyak Stepsize Violation Detector" (PSVD) problem. In Subsection~\ref{subsection 2.1}, we highlight the key ideas behind the PSVD problem and the level-adjustment formula, followed by the pseudo-code of the algorithm. In Subsection~\ref{subsec_ituitive_explanation}, we provide an intuitive explanation of our method using schematic diagrams. In Subsection~\ref{subsection 2.2}, we rigorously prove the convergence of the proposed approach.

\subsection{Polyak Stepsize Violation Detector Problem for Level Adjustment} 
\label{subsection 2.1}
As discussed in the Introduction, the Polyak stepsize \eqref{polyak stepsize with optimal obj} exhibits favorable convergence properties. However, these properties may not hold when the optimal value is inaccurately estimated using a level value. In this subsection, we introduce a framework that identifies such inaccuracies through a novel decision-guided approach. We begin by presenting the core concepts in Lemmas \ref{first_lemma_introducing_idea} and \ref{second_lemma_introducing_idea}.

\begin{lemma}
\label{first_lemma_introducing_idea}
For $x^k\in \mathcal{X}$ and subgradient $g^k$ at $x^k$, if a stepsize $s^{k}$ does not satisfy the following inequality:
\begin{equation}
    \left(g^{k}\right)^{\top} x^{\star} \leq\left(g^{k}\right)^{\top} x^{k}- \frac{1}{\bar{\gamma}} s^{k}\left\|g^{k}\right\|^2,
    \label{first_lemma_introducing_idea-1}
\end{equation}
where $x^{\star} \in \mathcal{X}^{\star}$, then 
\begin{equation}
    s^k > \bar{\gamma} \cdot \frac{f\left(x^k\right)-f^{\star}}{\left\|g^k\right\|^2}.
    \label{first_lemma_introducing_idea-2}
\end{equation}
\end{lemma}

\begin{proof}
    From the violation of  \eqref{first_lemma_introducing_idea-1}, it follows that
    \begin{equation}
         s^{k}> \frac{\bar{\gamma}\cdot \left(g^{k}\right)^{\top} (x^{k}-x^{\star})}{\left\|g^{k}\right\|^2}.
    \end{equation}
Together with \eqref{subgradient inequality}, we have \eqref{first_lemma_introducing_idea-2}. The proof is complete.
\end{proof} 

The above lemma indicates that inequality \eqref{first_lemma_introducing_idea-1} can be used to detect whether a stepsize exceeds $\bar{\gamma} (f(x^k) - f^{\star})/\|g^k\|^2$. Therefore, based on the above lemma, inequality \eqref{first_lemma_introducing_idea-1} serves as a test to determine whether the level value is inappropriately low and can be used to trigger a level adjustment. In the following lemma, we show that a new level value can be constructed whenever inequality \eqref{first_lemma_introducing_idea-1} is violated.

\begin{lemma}
\label{second_lemma_introducing_idea}
Suppose that at iteration $k$, $x^{k}$ is the current solution, $g^{k}$ is a subgradient of $f$ at $x^{k}$, and the stepsize $s^k$ is computed using a level value $\bar{f}_k$ (with $\bar{f}_k<f^{\star}$) according to \eqref{polyak stepsize with estimation}. If inequality \eqref{first_lemma_introducing_idea-1} is violated, then a new level value $\bar{f}^{\prime}$ can be generated as
\begin{equation} 
    \bar{f}^{\prime}=\frac{\gamma}{\bar{\gamma}} \bar{f_k}+\left(1-\frac{\gamma}{\bar{\gamma}}\right) f\left(x^k\right),
\label{second_lemma_introducing_idea-1}
\end{equation}
which serves as a tighter lower bound on $f^{\star}$, satisfying
\begin{equation}
    \bar{f_k} < \bar{f}^{\prime} < f^{\star}.
\end{equation}
\end{lemma}

\begin{proof}
\noindent From \eqref{second_lemma_introducing_idea-1}, $\bar{f}^{\prime}$ is a convex combination of $\bar{f}_k$ and $f\left(x^k\right)$. Since $f(x^k) > \bar{f}_k$, it follows directly that $\bar{f}^{\prime} > \bar{f}_k$. 

By Lemma \ref{first_lemma_introducing_idea}, the violation of \eqref{first_lemma_introducing_idea-1} implies that \eqref{first_lemma_introducing_idea-2} holds. Substituting the stepsize from \eqref{polyak stepsize with estimation} into \eqref{first_lemma_introducing_idea-2}, we obtain:
\begin{equation}
    \gamma \cdot \frac{f\left(x^k\right)-\bar{f}_k}{\left\|g^k\right\|^2} > \bar{\gamma} \cdot \frac{f\left(x^k\right)-f^{\star}}{\left\|g^k\right\|^2}.
\end{equation}
Canceling $\|g^k\|^2$ and rearranging terms yields:
\begin{equation}
    f^{\star} > f\left(x^k\right)-\frac{\gamma}{\bar{\gamma}} f\left(x^k\right)+\frac{\gamma}{\bar{\gamma}} \bar{f}_k,
\end{equation}
and the right-hand side is exactly $\bar{f}^{\prime}$. Thus, we have $\bar{f}^{\prime} < f^{\star}$. The proof is complete. 
\end{proof}

The above two lemmas demonstrate that if a violation of \eqref{first_lemma_introducing_idea-1} is detected at iteration $k$, then the stepsize exceeds $\bar{\gamma} (f(x^k) - f^{\star})/\|g^k\|^2$, and a new level value is generated using the convex combination \eqref{second_lemma_introducing_idea-1}. However, directly checking \eqref{first_lemma_introducing_idea-1} is challenging since $x^{\star}$ is unknown. In the following theorem, we propose a novel ``Polyak Stepsize Violation Detector" (PSVD) constraint satisfaction problem that circumvents the need for $x^{\star}$.
Specifically, at a given iteration, the problem formulation only requires solutions, subgradients, and stepsizes from previous iterations. We show that the infeasibility of the problem implies the violation of \eqref{first_lemma_introducing_idea-1} and therefore triggers a level adjustment. 


\begin{theorem}
\label{Theorem 2.3}
Suppose that, from iteration $k(j)$ through iteration $k(j)+\eta$, the level values remain constant and underestimate $f^{\star}$, and $\{x^k\}_{k=k(j)}^{k(j)+\eta}$ is generated iteratively using subgradients $\{g^k\}_{k=k(j)}^{k(j)+\eta}$ and stepsizes $\{s^k\}_{k=k(j)}^{k(j)+\eta}$. Consider the following PSVD problem formulated as a constraint satisfaction problem with decision vector $x \in \mathbb{R}^n$:
\begin{equation}
\label{PSVD problem}
    \left\{\begin{array}{c} \begin{aligned}
    \left(g^{k(j)}\right)^{\top} x & \leq\left(g^{k(j)}\right)^{\top} x^{k(j)}-\frac{1}{\bar{\gamma}} s^{k(j)}\left\|g^{k(j)}\right\|^2, \\
    \left(g^{k(j)+1}\right)^{\top} x & \leq\left(g^{k(j)+1}\right)^{\top} x^{k(j)+1}-\frac{1}{\bar{\gamma}} s^{k(j)+1}\left\|g^{k(j)+1}\right\|^2, \\
    & \vdots \\
    \left(g^{k(j)+\eta}\right)^{\top} x & \leq\left(g^{k(j)+\eta}\right)^{\top} x^{k(j)+\eta}-\frac{1}{\bar{\gamma}} s^{k(j)+\eta}\left\|g^{k(j)+\eta}\right\|^2. 
\end{aligned}\end{array}\right.\end{equation}
If \eqref{PSVD problem} is infeasible, then there exists $\kappa \in \{k(j), k(j)+1, \dots, k(j)+\eta\}$, such that
\begin{equation}
    s^\kappa > \bar{\gamma} \cdot \frac{f\left(x^\kappa\right)-f^{\star}}{\left\|g^\kappa\right\|^2},\label{T2.3_1}
\end{equation}
and a new level value $\bar{f}^{\prime}$ can be computed as
\begin{equation}
    \bar{f}^{\prime}=\frac{\gamma}{\bar{\gamma}} \bar{f}_{k(j)+\eta}+\left(1-\frac{\gamma}{\bar{\gamma}}\right) \min _{\kappa \in\{k(j), k(j)+1, \dots, k(j)+\eta\}}f\left(x^\kappa\right),
\label{level updating for subgradient method}
\end{equation}
which is a tighter level as compared to the previous level value $\bar{f}_{k(j)+\eta}$, i.e., 
\begin{equation}
    \bar{f}_{k(j)+\eta} < \bar{f}^{\prime} <  f^{\star}.
\end{equation}
\end{theorem}

\begin{proof} 
Since \eqref{PSVD problem} admits no feasible solution,  $x^{\star}\in\mathcal{X}^{\star}$ is not feasible to \eqref{PSVD problem}. Therefore, there exists $\kappa \in \{k(j), k(j)+1, \dots, k(j)+\eta\}$ such that the following inequality holds:
\begin{equation}
\left(g^{\kappa}\right)^{\top} x^{\star} > \left(g^{\kappa}\right)^{\top} x^{\kappa}- \frac{1}{\bar{\gamma}} s^{\kappa}\left\|g^{\kappa}\right\|^2.
\end{equation}
By Lemma \ref{first_lemma_introducing_idea}, this implies  \eqref{T2.3_1}. Next, applying Lemma \ref{second_lemma_introducing_idea}, we obtain
\begin{equation}
    \frac{\gamma}{\bar{\gamma}} \bar{f}_{k(j)+\eta}+\left(1-\frac{\gamma}{\bar{\gamma}}\right) f\left(x^{\kappa}\right) < f^{\star}.
\end{equation}
Consequently,
\begin{equation}
    \bar{f}^{\prime}=\frac{\gamma}{\bar{\gamma}} \bar{f}_{k(j)+\eta}+\left(1-\frac{\gamma}{\bar{\gamma}}\right) \min _{\kappa \in\{k(j), \dots, k(j)+\eta\}}\left\{f\left(x^\kappa\right)\right\} < f^{\star}. 
\end{equation}
Since $\gamma<\bar{\gamma}$, equation \eqref{level updating for subgradient method} is a convex combination.  Note that $\bar{f}_{k(j)+\eta}$ is less than $\min_{\kappa \in \{k(j),\ldots,k(j)+\eta\}}{f(x^\kappa)}$, and therefore $\bar{f}^{\prime} > \bar{f}_{k(j)+\eta}$. The proof is complete. 
\end{proof}

The aforementioned results remain valid when any inequality satisfied by $x^{\star}\in\mathcal{X}^{\star}$ is added to the PSVD problem \eqref{PSVD problem}. Introducing additional inequalities makes the PSVD problem tighter, which may lead to more frequent level adjustments. 
For example, in the dual problems of integer programming discussed in Subsection~\ref{subsection 4.2.1}, each element of a feasible solution $x$ must be non-negative. Therefore, the constraint $x\geq 0$ can be incorporated. 

According to the above theorem, the level value can be dynamically adjusted by periodically checking the feasibility of the PSVD problem \eqref{PSVD problem}. When infeasibility is detected, a new level value is computed. The complete algorithm is presented below.

\noindent \textbf{Algorithm 2.1: Subgradient Method with PSADLA} 

\noindent\textbf{Step 1 (Initialization)} Initialize $x^0$, initialize a level value $\bar{f}_0 < f^{\star}$, and select parameters $\gamma$ and $\Bar{\gamma}$ such that $0 < \gamma < \Bar{\gamma} < 2$. Set $k = 0$.

\noindent\textbf{Step 2 (Update)} Update as
$x^{k+1}=P_{\mathcal{X}}\left(x^k-s^k \cdot g^k\right)$,
where the stepsize is given by $s^k=\gamma \cdot \frac{f\left(x^k\right)-\bar{f}_{k}}{\left\|g^k\right\|^2}$.

\noindent\textbf{Step 3 (Level Adjustment)} Append $\left(g^{k}\right)^{\top} x \leq\left(g^{k}\right)^{\top} x^{k}-\frac{1}{\bar{\gamma}} s^{k}\left\|g^{k}\right\|^2$ to the PSVD problem \eqref{PSVD problem}. If the resulting problem admits no feasible solution, update the level value using the rule \eqref{level updating for subgradient method} and remove all previously added inequalities from the PSVD problem. Otherwise, keep the same level, i.e., $\bar{f}_{k+1} = \bar{f}_{k}$.

\noindent\textbf{Step 4 (Stopping Criteria Checking)} Stop if any of the following is satisfied: the gap between the best objective function value found so far and the current level is sufficiently small, i.e., $\min _{\kappa \in\{1, 2, \dots, k\}}f\left(x^{\kappa}\right) - \bar{f}_{k}< \varsigma$; the maximum number of iterations is reached; the allotted computational time is exceeded. If none of these conditions is met, increment $k$ by 1 and return to Step 2.

In the following, we compare the PSVD problem with the SDD problem \eqref{divergencedetection}. We begin by stating a proposition that provides the basis for comparing the relative ``tightness" of the two problems.

\begin{proposition} 
\label{Prop 2.1}
If $x^{\prime}$ is a feasible solution of the PSVD problem \eqref{PSVD problem}, it is also feasible to the following constraint satisfaction problem with decision vector $x \in \mathbb{R}^n$:
\begin{equation}\label{prop2.1_1}
    \left\{\begin{aligned}
    \left\|x-x^{k(j)+1}\right\|^2-\left\|x-x^{k(j)}\right\|^2 & \leq\left(1-\frac{2}{\bar{\gamma}}\right)\left(s^{k(j)}\right)^2\left\|g^{k(j)}\right\|^2, \\
    \left\|x-x^{k(j)+2}\right\|^2-\left\|x-x^{k(j)+1}\right\|^2 & \leq\left(1-\frac{2}{\bar{\gamma}}\right)\left(s^{k(j)+1}\right)^2\left\|g^{k(j)+1}\right\|^2, \\
    \vdots & \\
    \left\|x-x^{k(j)+\eta+1}\right\|^2-\left\|x-x^{k(j)+\eta}\right\|^2 & \leq\left(1-\frac{2}{\bar{\gamma}}\right)\left(s^{k(j)+\eta}\right)^2\left\|g^{k(j)+\eta}\right\|^2.
    \end{aligned}\right.\end{equation}
\end{proposition}
\begin{proof}
Consider the first inequality of the PSVD problem \eqref{PSVD problem} and the first inequality of \eqref{prop2.1_1}.
The first inequality of \eqref{PSVD problem} can equivalently be written as 
\begin{equation}\label{prop2.1_2}
    -2 s^{k(j)}\left(x^{k(j)}-x\right)^{\top} g^{k(j)}+\left(s^{k(j)}\right)^2\left\|g^{k(j)}\right\|^2 \leq\left(1-\frac{2}{\bar{\gamma}}\right)\left(s^{k(j)}\right)^2\left\|g^{k(j)}\right\|^2.
\end{equation}
By the non-expansive property of $P_{\mathcal{X}}$, we have the following inequality for $\forall x \in \mathcal{X}$:
\begin{equation}\label{prop2.1_3}
    \left\|x^{k(j)+1}\!-\!x\right\|^2\!-\!\left\|x^{k(j)}\!-\!x\right\|^2 \leq -2 s^{k(j)}\left(x^{k(j)}\!-\!x\right)^{\top} g^{k(j)}\!+\!\left(s^{k(j)}\right)^2\left\|g^{k(j)}\right\|^2.
\end{equation}
The right-hand side of \eqref{prop2.1_3} is the same as the left-hand side of \eqref{prop2.1_2}. If \eqref{prop2.1_2} holds for some $x^{\prime}$, then we must have 
\begin{equation}
    \left\|x^{k(j)+1}-x^{\prime}\right\|^2-\left\|x^{k(j)}-x^{\prime}\right\|^2 \leq \left(1-\frac{2}{\bar{\gamma}}\right)\left(s^{k(j)}\right)^2\left\|g^{k(j)}\right\|^2,
\end{equation}
which implies that $x^{\prime}$ satisfies the first inequality of \eqref{prop2.1_1}. The results hold for other inequalities. The proof is complete.
\end{proof}

In the above proposition, it is shown that the PSVD problem \eqref{PSVD problem} is equivalent to or is tighter than the problem \eqref{prop2.1_1}; that is, the feasible region of the PSVD problem \eqref{PSVD problem} is equal to or is contained within that of problem \eqref{prop2.1_1}. Moreover, the left-hand side of \eqref{prop2.1_1} is identical to that of \eqref{divergencedetection}, while its right-hand side is negative, implying that the problem \eqref{prop2.1_1} is strictly tighter than the SDD problem \eqref{divergencedetection}. Therefore, we conclude that the PSVD problem is strictly tighter than the SDD problem \eqref{divergencedetection}. Therefore, with the PSVD problem, the level value can be adjusted more frequently. Furthermore, the PSVD problem \eqref{PSVD problem} is linear and can be solved using linear programming methods. Therefore, checking its feasibility is computationally more efficient than that of the SDD problem.

\subsection{Intuitive Explanation}
\label{subsec_ituitive_explanation}

In this subsection, we use diagrams to illustrate the intuition behind our level adjustment approach. The PSVD problem is developed to detect whether a stepsize $s_k$ violates the following inequality:
\begin{equation}
\label{eq_polyak_inequality}
    0<s^k<\bar{\gamma}\cdot \frac{f(x^k)-f^{\star}}{||g^k||^2},
\end{equation}
where parameter $\bar{\gamma}<2$. In Fig. \ref{Illustration_figure_sub1}, a solution $x^{k}\in\mathcal{X}$, the subgradient $g^k$ at $x^{k}$, and an optimal solution $x^{\star}\in\mathcal{X}^{\star}$ are depicted. Due to convexity, the negative direction of the subgradient (the red dashed line) must form an acute angle with the direction from $x^k$ to $x^\star$ (the black dashed line). The projection of $x^{\star} - x^k$ onto the normalized negative subgradient direction has a length of $\frac{(-g^k)^{\top}(x^{\star}-x^{k})}{\|g^k\|}$. In Lemma \ref{first_lemma_introducing_idea}, it has been shown that if the level value is proper such that \eqref{eq_polyak_inequality} is satisfied, then $s^k\|g^k\|/\bar{\gamma}$ is less than or equal to the projection, i.e., 
\begin{equation}
\label{relationship_projection_stepsize}
    \frac{s^k\|g^k\|}{\bar{\gamma}}\leq \frac{(-g^k)^{\top}(x^{\star}-x^{k})}{\|g^k\|}.
\end{equation}
Also, if inequality \eqref{relationship_projection_stepsize} is violated, then the stepsize is too large, causing a violation of \eqref{eq_polyak_inequality}. Therefore, the violation of \eqref{relationship_projection_stepsize} indicates that the current level value is inappropriate. Moreover, we have shown that a ``better'' level value (closer to the actual optimal value) can be calculated using \eqref{second_lemma_introducing_idea-1} whenever such a violation is detected.

\begin{figure}[h!]
    \centering
    \begin{subfigure}[b]{0.6\textwidth}
        \includegraphics[width=\textwidth]{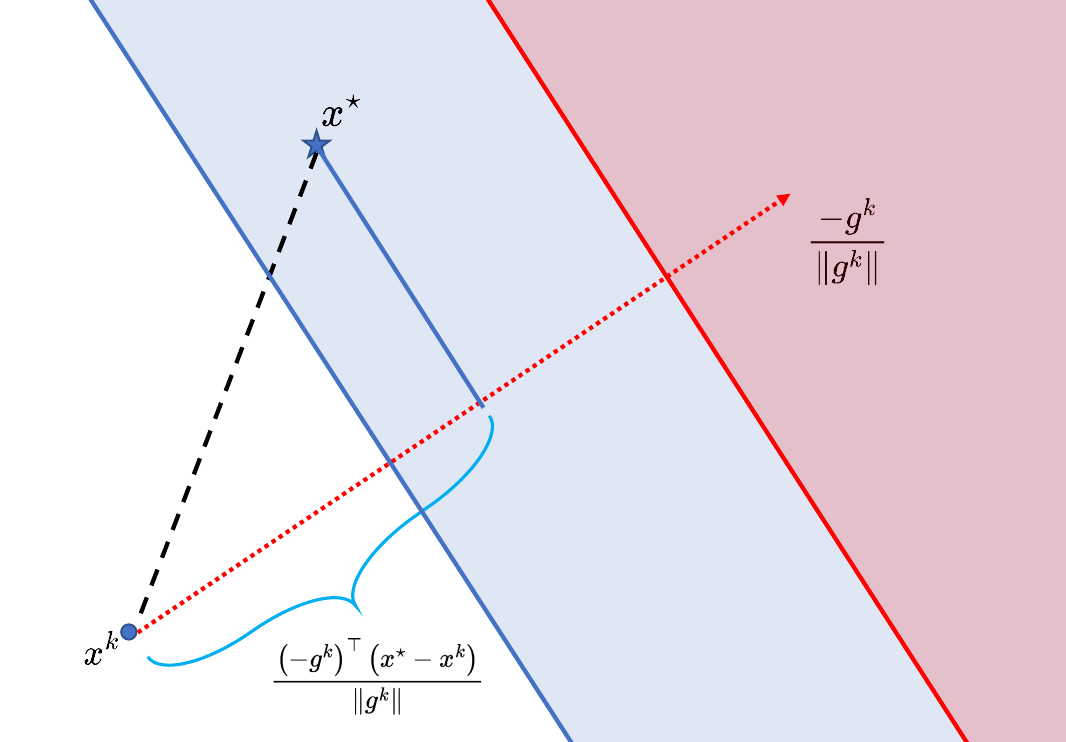}
        \caption{Illustration of the half-spaces for establishing a PSVD problem}
        \label{Illustration_figure_sub1}
    \end{subfigure}
    \hfill
    \begin{subfigure}[b]{0.6\textwidth}
        \includegraphics[width=\textwidth]{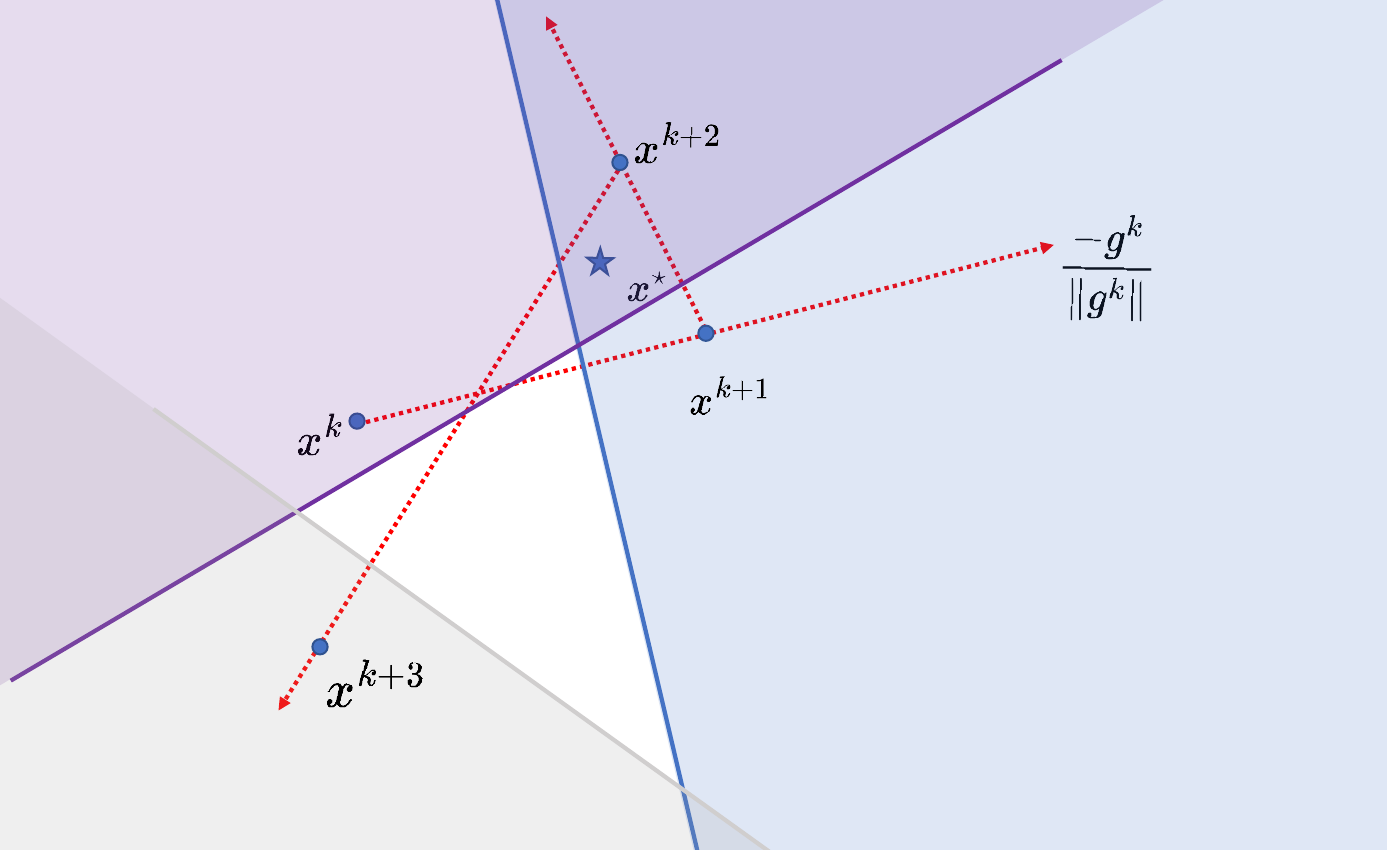}
        \caption{A PSVD problem established in three iterations}
        \label{Illustration_figure_sub2}
    \end{subfigure}
    \caption{Illustration of the PSVD-based level adjustment}
    \label{Illustration_half_spaces}
\end{figure}

Since $x^{\star}$ is unknown, it is impractical to check the violation of \eqref{relationship_projection_stepsize} directly. However, we can define the following half-space based on \eqref{relationship_projection_stepsize} as
\begin{equation}
   H^k=\left\{ x\in\mathbb{R}^n : \frac{(-g^k)^{\top}(x-x^{k})}{\|g^k\|} \geq \frac{s^k\|g^k\|}{\bar{\gamma}} \right\}.
\end{equation}
If $s^k$ satisfies \eqref{eq_polyak_inequality}, then the corresponding half-space $H^k$ ``covers'' the optimal solution $x^{\star}$ as illustrated by the blue-shaded region in Fig.~\ref{Illustration_figure_sub1}. 
Conversely, if $H^k$ does not cover $x^{\star}$, as illustrated by the red-shaded region in Fig. \ref{Illustration_figure_sub1}, then \eqref{eq_polyak_inequality} is violated. Building on this idea, we collect subgradients and stepsizes from a sequence of iterations and establish a series of half-spaces to define the PSVD problem. If the level value used is ``proper'' such that \eqref{eq_polyak_inequality} is satisfied, then the optimal solution $x^{\star}$ must be covered by all the half-spaces (their intersection is not empty), and therefore the PSVD must admit at least one feasible solution such as $x^{\star}$ itself. However, if the PSVD problem is infeasible, then at least one inequality \eqref{relationship_projection_stepsize} must have been violated, implying that the level value is too low and needs to be adjusted. Following Theorem \ref{Theorem 2.3}, a better level value can be generated. Furthermore, we will establish a key result in Lemma~\ref{Lemma-subgradient-infinitely-often}: for any given level value, the PSVD problem becomes infeasible within a finite number of iterations. 

To visualize the violation of the PSVD problem over multiple iterations, Fig. \ref{Illustration_figure_sub2} illustrates a PSVD problem constructed across three iterations. Assume that stepsizes at each iterations are calculated by using the same level value. At iteration $k$, the level value is proper, \eqref{relationship_projection_stepsize} is satisfied, and the corresponding half-space (the blue region) contains the optimal solution $x^{\star}$. The updated solution $x^{k+1}$ moves closer to $x^{\star}$. At iteration $k+1$, the level remains proper, and the corresponding half-space (the purple region) also contains $x^{\star}$. However, at iteration $k+2$, the level value is detected as ``inappropriate,'' and \eqref{relationship_projection_stepsize} is violated. The corresponding half-space excludes $x^{\star}$, and the intersection of all three half-spaces is empty.  As a result, the PSVD problem defined by the three inequalities corresponding to the three half-spaces has no feasible solution, thus triggering a level adjustment.

\subsection{Convergence Proof}
\label{subsection 2.2}

In this subsection, we prove the convergence of the subgradient method with PSADLA under the subgradient boundedness assumption. We begin by presenting a key lemma. 

\begin{lemma}
\label{Lemma-subgradient-infinitely-often} 
    In the subgradient method with PSADLA, the PSVD problem \eqref{PSVD problem} becomes infeasible infinitely often as $k \rightarrow \infty$.
\end{lemma}
\begin{proof} 
We proceed by contradiction. Suppose that, after some iteration $k^{\prime}$, the problem \eqref{PSVD problem} always admits a feasible solution $x^{\prime}\in \mathbb{R}^n$; that is, 
\begin{equation}
    \left(g^{\kappa}\right)^{\top} x^{\prime} \leq\left(g^{\kappa}\right)^{\top} x^{\kappa}-\frac{1}{\bar{\gamma}} s^{\kappa}\left\|g^{\kappa}\right\|^2, \quad \forall \kappa\in\{k^{\prime}, \dots, k\}.
\end{equation}
By Proposition \ref{Prop 2.1}, $x^{\prime}$ also satisfies
\begin{equation}
     \left\|x^{\prime}-x^{\kappa+1}\right\|^2-\left\|x^{\prime}-x^{\kappa}\right\|^2 \leq\left(1-\frac{2}{\bar{\gamma}}\right)\left(s^{\kappa}\right)^2\left\|g^{\kappa}\right\|^2,
     \quad \forall \kappa\in\{k^{\prime}, \dots, k\}.
\end{equation}
Summing these inequalities and then letting $k \rightarrow \infty$, we arrive at
\begin{equation}
    \left\|x^{\prime}-x^{\infty}\right\|^2-\left\|x^{\prime}-x^{k^{\prime}}\right\|^2 \leq \sum_{\kappa=k'}^{\infty} \left(1-\frac{2}{\bar{\gamma}}\right)\left(s^{\kappa}\right)^2\left\|g^{\kappa}\right\|^2. \\
\end{equation}
Since $\left\|x^{\prime}-x^{\infty}\right\|^2$ is positive, the following inequality holds 
\begin{equation}
    \left(\frac{2}{\bar{\gamma}}-1\right)\sum_{\kappa=k'}^{\infty} \left(s^{\kappa}\right)^2\left\|g^{\kappa}\right\|^2
    \leq 
    \left\|x^{\prime}-x^{k^{\prime}}\right\|^2.
\end{equation}
 The right-hand side is bounded, implying that the series on the left side must be summable. Therefore,
\begin{equation}\label{Lemma-subgradient-infinitely-often contradiction part 1}
    \lim _{\kappa \rightarrow \infty}\left(s^\kappa\right)^2\left\|g^\kappa\right\|^2=0.
\end{equation}

We now derive a contradiction. Since the problem remains feasible for all iterations after $k^{\prime}$, no level adjustment occurs beyond this point. Let $\bar{f}^{\prime}$ denote the last updated level value. Then, there exists a positive $\beta$ such that 
\begin{equation}
    f(x^k)-\bar{f}^{\prime} \geq \beta, \quad \forall k \geq k^{\prime}.
\end{equation}
Under Assumption \ref{Assumption_1}, which guarantees bounded subgradients, we have 
\begin{equation}  \left(s^k\right)^2\left\|g^k\right\|^2= \frac{\left(f\left(x^k\right)-\bar{f}^{\prime}\right)^2}{\left\|g^k\right\|^2} \geq \frac{\beta^2}{C^2}, \quad \forall k \geq k^{\prime}.
\end{equation}
This contradicts \eqref{Lemma-subgradient-infinitely-often contradiction part 1}. The proof is complete. 
\end{proof}

The above lemma reveals an important fact: the PSVD problem must become infeasible within a finite number of iterations. Therefore, as $k\rightarrow \infty$, the PSVD problem becomes infeasible infinitely often, thereby triggering an infinite number of level adjustments. This lemma plays a central role in establishing the convergence of both the level value and the value of the objective function.

\begin{theorem}
\label{Theorem 2.4} 
As $k \rightarrow \infty$, the level values converge to $f^{\star}$, i.e.,
\begin{equation} \label{section2_limit_of_level}
    \lim _{k \rightarrow \infty}\bar{f}_{k} = f^{\star}.
\end{equation}
Moreover, we have
\begin{equation} \label{limit of f - section 2}
    \lim _{k \rightarrow \infty}\Big\{\inf _{\kappa \leq k} f\left(x^\kappa\right)\Big\}=f^{\star}.
\end{equation}\end{theorem}

\begin{proof}
\textbf{Proof for \eqref{section2_limit_of_level}:} According to Lemma \ref{Lemma-subgradient-infinitely-often}, the level value is adjusted infinitely often as $k \rightarrow \infty$. Since the level value remains unchanged when the PSVD problem is feasible, to establish \eqref{section2_limit_of_level}, it suffices to show that
$\bar{f}_{k(j)}\rightarrow f^{\star}$ as $j \rightarrow \infty$. The sequence $\{\bar{f}_{k(j)}\}_{j=1}^{\infty}$ increases monotonically and, by Theorem \ref{Theorem 2.3}, is upper bounded by $f^{\star}$; thus, it must converge to a finite limit.  We now prove by contradiction that this limit must be $f^{\star}$. 

Assume the limit is not $f^{\star}$. Then there exists $\rho>0$ and an index $j^{\prime}$ such that for all $j\geq j^{\prime}$, 
\begin{equation}
    f^{\star}-\bar{f}_{k(j)} \geq \rho.\label{T2.4 2}
\end{equation}
From the level adjustment formula \eqref{level updating for subgradient method}, for any two successive elements within $\{\bar{f}_{k(j)}\}_{j=1}^{\infty}$, the following holds:
\begin{equation}
\label{T2.4 3}
    \bar{f}_{k(j+1)}=\frac{\gamma}{\bar{\gamma}} \bar{f}_{k_{j}}+\left(1-\frac{\gamma}{\bar{\gamma}}\right) \min _{\kappa \in\{k(j), k(j)+1, \dots, k(j+1)-1\}}f\left(x^{\kappa}\right).
\end{equation}
Subtracting  $-\bar{f}_{k(j)}$ from both sides leads to
\begin{equation}
\begin{aligned}
    &\bar{f}_{k(j+1)}-\bar{f}_{k(j)}=\\
    &-\left(1-\frac{\gamma}{\bar{\gamma}}\right) \bar{f}_{k(j)}+\left(1-\frac{\gamma}{\bar{\gamma}}\right) \min _{\kappa \in\{k(j), k(j)+1, \dots, k(j+1)-1\}}f\left(x^{\kappa}\right).
\end{aligned}
\end{equation}
Since \eqref{T2.4 2} holds, for $j\geq j^{\prime}$, it follows that 
\begin{equation} 
\label{different of levels}
\begin{aligned}
    &\bar{f}_{k(j+1)}-\bar{f}_{k(j)}=\\
    &\left(1-\frac{\gamma}{\bar{\gamma}}\right)\left(\min _{\kappa \in\{k(j), k(j)+1, \dots, k(j+1)-1\}}\left\{f\left(x^{\kappa}\right)\right\}-\bar{f}_{k(j)}\right) \geq\left(1-\frac{\gamma}{\bar{\gamma}}\right) \rho.
\end{aligned}
\end{equation}
According to Lemma \ref{Lemma-subgradient-infinitely-often}, the level value is adjusted infinitely often. Summing \eqref{different of levels} over $j = j^{\prime},\dots,\infty$ 
leads to
\begin{equation}
    \bar{f}_{\infty}-\bar{f}_{k(j^\prime)} = \infty .
\end{equation}    
This contradicts the fact that the sequence $\{\bar{f}_{k(j)}\}_{j={j^{\prime}}}^{\infty}$ is bounded above by $f^{\star}$. Hence, we must have $\lim _{j \rightarrow \infty}\bar{f}_{k(j)} = f^{\star}$. This proves \eqref{section2_limit_of_level}.

\textbf{Proof for \eqref{limit of f - section 2}: } 
Letting  $j\rightarrow \infty$ in \eqref{T2.4 3}, we obtain:
\begin{equation}
    \lim _{j \rightarrow \infty} \bar{f}_{k(j+1)}=\frac{\gamma}{\bar{\gamma}} \lim _{j \rightarrow \infty} \bar{f}_{k(j)}+\left(1-\frac{\gamma}{\bar{\gamma}}\right) \lim _{j \rightarrow \infty}\left\{ \min _{\kappa \in\{k(j), k(j)+1, \dots, k(j+1)-1\}}f\left(x^{\kappa}\right)\right\}.
\end{equation}
Since we have shown that  $\lim _{j \rightarrow \infty}\bar{f}_{k(j)} = f^{\star}$, we have 
\begin{equation}
    f^{\star}=\frac{\gamma}{\bar{\gamma}} f^{\star}+\left(1-\frac{\gamma}{\bar{\gamma}}\right) \lim _{j \rightarrow \infty}\left\{  \min _{\kappa \in\{k(j), k(j)+1, \dots, k(j+1)-1\}}f\left(x^{\kappa}\right)\right\},
\end{equation}
which simplifies to:
\begin{equation}
    \lim _{j \rightarrow \infty} \Big\{ \min _{\kappa \in\{k(j), k(j)+1, \dots, k(j+1)-1\}}f\left(x^{\kappa}\right)\Big\}=f^{\star}.
\end{equation}
Thus, there exists a subsequence of $\{f(x^{k})\}$ converges to $f^{\star}$. Since $\{f(x^{k})\}$ is bounded below by $f^{\star}$, it follows that
\begin{equation}
    \lim _{k \rightarrow \infty}\Big\{\inf _{\kappa \leq k} f\left(x^\kappa\right)\Big\}=f^{\star}. 
\end{equation}
This completes the proof. 
\end{proof}

\section{Convergence with Approximate Subgradients} \label{section 3}

For some optimization problems, computing an exact subgradient in each iteration is computationally expensive. In this section, we present a significant and insightful result: with our decision-guided level adjustment scheme, it is not necessary to compute an exact subgradient at every iteration—an approximate subgradient suffices to ensure convergence. Specifically, we define $\tilde{g}^k\in \mathbb{R}^n$ to be an approximate subgradient at iteration $k$, if there exists a scalar $F^k$ such that $\bar{f}_{k} < F^k \leq f(x^k)$, and the following inequality holds
\begin{equation}
\label{approximate subgradient definition}
    f^{\star}-F^k \geq\left(\tilde{g}^k\right)^{\top}\left(x^{\star}-x^k\right), \forall x^{\star} \in \mathcal{X}^{\star}.
\end{equation}
Compared to the original subgradient requirement in~\eqref{subgradient inequality}, this condition is significantly more relaxed. Notably, the left-hand side of \eqref{approximate subgradient definition} is larger, and the inequality only needs to hold for all $x^{\star} \in \mathcal{X}^{\star}$, rather than for all  $x \in \mathcal{X}$. The computation of $\tilde{g}^k$ and $F^k$ is problem-specific, and is discussed in Subsection~\ref{subsection 3.2}.

With the approximate subgradient $\tilde{g}^k$, the solution is updated at iteration $k$ as 
\begin{equation}
\label{update equation with surrogate subgradient} 
x^{k+1} = P_\mathcal{X}\left(x^k - s^k \tilde{g}^k\right), \end{equation}
with
\begin{equation}
\label{polyak stepsize surrogate subgradient} 
s^k = \gamma \cdot \frac{F^k - \bar{f}_{k}}{\left\|\tilde{g}^k\right\|^2}, \quad 0 < \gamma < \bar{\gamma} < 2.
\end{equation}
As in the standard subgradient method, the level value is  adjusted by periodically checking the feasibility of a PSVD problem. Specifically, the PSVD problem at iteration $k(j)+\eta$ is defined by the following set of inequalities:
\begin{equation}
\label{PSVD for surrogate subgradient}
    \left\{\begin{aligned}
    \left(\tilde{g}^{k(j)}\right)^{\top} x &\leq\left(\tilde{g}^{k(j)}\right)^{\top} x^{k(j)}-\frac{1}{\bar{\gamma}} s^{k(j)}\left\|\tilde{g}^{k(j)}\right\|^2, \\
    \left(\tilde{g}^{k(j)+1}\right)^{\top} x &\leq\left(\tilde{g}^{k(j)+1}\right)^{\top} x^{k(j)+1}-\frac{1}{\bar{\gamma}} s^{k(j)+1}\left\|\tilde{g}^{k(j)+1}\right\|^2, \\
    \vdots \\
    \left(\tilde{g}^{k(j)+\eta}\right)^{\top} x &\leq\left(\tilde{g}^{k(j)+\eta}\right)^{\top} x^{k(j)+\eta}-\frac{1}{\bar{\gamma}} s^{k(j)+\eta}\left\|\tilde{g}^{k(j)+\eta}\right\|^2.
\end{aligned}\right.\end{equation}
If this problem admits no feasible solution, then a new level value is generated as 
\begin{equation}
\label{level adjustment formula surrogate subgradient}
    \bar{f}^{\prime}=\frac{\gamma}{\bar{\gamma}} \bar{f}_{k(j)+\eta}+\left(1-\frac{\gamma}{\bar{\gamma}}\right) \min _{\kappa \in\{k(j), \dots, k(j)+\eta\}}F^\kappa.
\end{equation}
The convergence results established in Section \ref{section 2} remain valid, under the boundedness assumption, $\|\tilde{g}^k\|\leq \tilde{C}, \forall k$, and the satisfaction of the following condition:
\begin{condition} 
\label{condition on Fk}
\noindent There exists a positive scalar $\varepsilon$, for any $k$, if $F^k \neq f(x^k)$, then $\bar{f}_{k}+\varepsilon \leq F^k$.
\end{condition} 
The above condition ensures that when $F^k$ is not strictly equal to $f(x^k)$, there is a gap between $F^k$ and the level value, ensuring a nonzero stepsize. This condition is important for establishing the convergence and is easy to satisfy. 

Compared to the subgradient method, the above approximate subgradient method may require more solution updates to achieve the same level of accuracy. However, this is not an issue, as calculating an approximate subgradient is generally much less computationally intensive than calculating an exact subgradient. For example, as discussed in Subsection \ref{subsection 3.2}, when solving the dual problem of an integer programming problem, the computational efficiency of calculating approximate subgradients significantly outweighs that of exact subgradients, making the approach advantageous in practice.

The convergence analysis is presented in Subsection~\ref{subsection 3.1}, followed by a discussion of practical applications in Subsection~\ref{subsection 3.2}. For two broad classes of problems, we provide detailed procedures for computing approximate subgradients and for satisfying Condition~\ref{condition on Fk}.

\subsection{Convergence Proof} \label{subsection 3.1}

This subsection establishes the convergence of the approximate subgradient method under the subgradient boundedness condition and Condition \ref{condition on Fk}. A theorem is presented next to introduce key ideas and to prove convergence. 

\begin{theorem}
\label{Theorem 3.1}
Suppose $\{x^k\}_{k=k(j)}^{k(j)+\eta}$ is generated iteratively according to \eqref{update equation with surrogate subgradient} using the approximate subgradients $\{\tilde{g}^k\}_{k=k(j)}^{k(j)+\eta}$,  stepsizes $\{s^k\}_{k=k(j)}^{k(j)+\eta}$, and identical level values $\{\bar{f}_{k}\}_{k=k(j)}^{k(j)+\eta}$. If the PSVD \eqref{PSVD for surrogate subgradient} is infeasible, then there exists $\kappa \in \{k(j), k(j)  + 1, \dots, k(j)+\eta\}$ such that
\begin{equation}
    s^\kappa > \bar{\gamma} \cdot \frac{F^\kappa-f^{\star}}{\left\|\tilde{g}^\kappa\right\|^2}.\label{T3.3_1}
\end{equation}
Then, a new level value $\bar{f}^{\prime}$ can be computed using \eqref{level adjustment formula surrogate subgradient} as
\begin{equation}
    \bar{f}^{\prime}=\frac{\gamma}{\bar{\gamma}} \bar{f}_{k(j)+\eta}+\left(1-\frac{\gamma}{\bar{\gamma}}\right) \min _{\kappa \in\{k(j), k(j)+1 \dots, k(j)+\eta\}}F^\kappa,
\end{equation}
which is a tighter level value, i.e., $\bar{f}_{k(j)+\eta}  <  \bar{f}^{\prime} < f^{\star}$.
\end{theorem}

\begin{proof}
The above theorem can be proven following the same idea as in Theorem \ref{Theorem 2.3}. Since the PSVD problem \eqref{PSVD for surrogate subgradient} admits no feasible solution, there must exist $\kappa \in \{k(j), k(j)+1, \dots, k(j)+\eta\}$ such that
 
\begin{equation}
    \left(\tilde{g}^{\kappa}\right)^{\top} x^{\star} > \left(\tilde{g}^{\kappa}\right)^{\top} x^{\kappa}-\frac{1}{\bar{\gamma}} s^{\kappa}\left\|\tilde{g}^{\kappa}\right\|^2,
\end{equation}
with $x^{\star} \in \mathcal{X}^{\star}$. 
By rearranging the terms of the above inequality, we have
\begin{equation}
     s^{\kappa}  \left\|\tilde{g}^{\kappa}\right\|^2 > \bar{\gamma} \cdot  \left(\tilde{g}^{\kappa}\right)^{\top} \left(x^{\kappa}-x^{\star}\right) .
\end{equation}
Since $\tilde{g}^{\kappa}$ satisfies \eqref{approximate subgradient definition}, the above inequality can be further written as
\begin{equation}
     s^{\kappa}  \left\|\tilde{g}^{\kappa}\right\|^2 > \bar{\gamma} \cdot  \left(\tilde{g}^{\kappa}\right)^{\top} \left(x^{\kappa}-x^{\star}\right) \geq
     \bar{\gamma} \cdot \left( F^\kappa-f^{\star} \right).
\end{equation}
This establishes inequality \eqref{T3.3_1}.

In the following, we will show $\bar{f}_{k(j)+\eta}<\bar{f}^{\prime}< f^{\star}$. Substituting  $s^\kappa$ in \eqref{T3.3_1} with the Polyak stepsize given by \eqref{polyak stepsize surrogate subgradient}, we obtain 
\begin{equation}
    f^{\star}> \left( 1- \frac{\gamma}{\bar{\gamma}}\right) F^\kappa + \frac{\gamma}{\bar{\gamma}} \bar{f}_{k(j)+\eta}.
\end{equation}
Since $\kappa \in\left\{k(j), \ldots, k(j)+\eta\right\}$, taking the minimum over all $\{k(j), \ldots, k(j)+\eta\}$ yields 
\begin{equation}
    f^{\star}> \left( 1- \frac{\gamma}{\bar{\gamma}}\right) \min _{\kappa \in\{k(j), \dots, k(j)+\eta\}}F^\kappa + \frac{\gamma}{\bar{\gamma}} \bar{f}_{k(j)+\eta},
\end{equation}
where the right-hand side term is $\bar{f}^{\prime}$. It is thus shown that $\bar{f}^{\prime} < f^{\star}$. Finally, since $\min _{\kappa \in\{k(j), \dots, k(j)+\eta\}}F^\kappa$ is strictly greater than the level value and \eqref{level adjustment formula surrogate subgradient} is a convex combination, it follows that
$\bar{f}_{k(j)+\eta} < \bar{f}^{\prime}$. This completes the proof.
\end{proof}

According to the above theorem, when using approximate subgradients, the PSVD problem can trigger the adjustment of the level value. The following proposition and lemma are instrumental in proving the convergence of level and objective function values. 
\begin{proposition} 
\label{Prop 3.1}
Any solution $x^{\prime} \in \mathbb{R}^n$ that is feasible for the PSVD problem \eqref{PSVD for surrogate subgradient} is also a feasible solution to the following constraint satisfaction problem:
\begin{equation}
\label{prop3.1_1}
    \left\{\begin{aligned}
    \left\|x-x^{k(j)+1}\right\|^2-\left\|x-x^{k(j)}\right\|^2 & \leq\left(1-\frac{2}{\bar{\gamma}}\right)\left(s^{k(j)}\right)^2\left\|\tilde{g}^{k(j)}\right\|^2, \\
    \left\|x-x^{k(j)+2}\right\|^2-\left\|x-x^{k(j)+1}\right\|^2 & \leq\left(1-\frac{2}{\bar{\gamma}}\right)\left(s^{k(j)+1}\right)^2\left\|\tilde{g}^{k(j)+1}\right\|^2, \\
    \vdots & \\
    \left\|x-x^{k(j)+\eta+1}\right\|^2-\left\|x-x^{k(j)+\eta}\right\|^2 & \leq\left(1-\frac{2}{\bar{\gamma}}\right)\left(s^{k(j)+\eta}\right)^2\left\|\tilde{g}^{k(j)+\eta}\right\|^2,
    \end{aligned}\right.
\end{equation}
\end{proposition}
where $x \in \mathbb{R}^n$ is a decision vector.

As the above proposition is a natural extension of Proposition \ref{Prop 2.1}, the proof is omitted for brevity.

\begin{lemma}
\label{infinitely_offen_surrogate_subgradient}
With approximate subgradients and PSADLA, the PSVD problem \eqref{PSVD for surrogate subgradient} becomes infeasible infinitely often as $k \rightarrow \infty$. 
\end{lemma}

\begin{proof} 
Assume, for the sake of contradiction, that there exists an iteration $k^{\prime} $ after which the PSVD problem \eqref{PSVD for surrogate subgradient} always remains feasible. In other words, for $k \geq k^{\prime}$, there exists $x^{\prime}\in \mathbb{R}^n$, which is feasible to the PSVD problem. 
From Proposition \ref{Prop 3.1}, the following inequalities hold
\begin{equation}
\label{lem3.1_1}
    \left\|x^{\prime}-x^{\kappa+1}\right\|^2-\left\|x^{\prime}-x^{\kappa}\right\|^2 \leq\left(1-\frac{2}{\bar{\gamma}}\right)\left(s^{\kappa}\right)^2\left\|\tilde{g}^{\kappa}\right\|^2, \forall \kappa\in\{k^{\prime}, \dots, k\}.
\end{equation}
Summing the above inequalities and letting $k\rightarrow \infty$, we obtain
\begin{equation}
    \left\|x^{\prime}-x^{\infty}\right\|^2-\left\|x^{\prime}-x^{k^{\prime}}\right\|^2 \leq\left(1-\frac{2}{\bar{\gamma}}\right)\lim _{k \rightarrow \infty}\left(s^{k}\right)^2\left\|\tilde{g}^{k}\right\|^2.
\end{equation}
Since $\left\|x^{\prime}-x^{\infty}\right\|^2$ is non-negative and $\|x^{\prime}-x^{k^{\prime}}\|^{2}$ is bounded above, it follows that
\begin{equation}\label{Lemma 3.1 contradiction part 1}
    \lim _{k \rightarrow \infty}\left(s^k\right)^2\left\|\tilde{g}^k\right\|^2=0.
\end{equation}
In the following, a contradiction is shown. Since the PSVD problem is feasible for $k\geq k^{\prime}$, the level value is not adjusted after $k^{\prime}$. Therefore, there exists a positive constant $\beta$ such that $f^{\star}-\bar{f}_k \geq \beta$ for all $k\geq k^{\prime}$. 
By the assumption that Condition \ref{condition on Fk} is satisfied, either $F^k=f(x^k)$ or $\bar{f}_{k}+\varepsilon \leq F^k $ is satisfied. In either case, it follows that:
\begin{equation}
    F^k-\bar{f}_k \geq \min(\varepsilon,\beta),\quad \forall k \geq k^{\prime}.
\end{equation}
With the above inequality and the boundedness assumption, the following inequality holds for any $k \geq k^{\prime}$:
\begin{equation}
    \left(s^k\right)^2\left\|\tilde{g}^k\right\|^2
    = \frac{\left(F^k-\bar{f}_k\right)^2}{\left\|\tilde{g}^k\right\|^2} 
    \geq
    \frac{\min(\varepsilon,\beta)^2}{\tilde{C}^2}.
\end{equation}
This implies that $(s^k)^2 \|\tilde{g}^k\|^2$ is bounded away from zero by $\frac{\min(\varepsilon,\beta)^2}{\tilde{C}^2}$.  However, this directly contradicts our earlier result that $\lim_{k \to \infty} (s^k)^2 \|\tilde{g}^k\|^2 = 0$.
The contradiction arises from the initial assumption that the problem remains always feasible after  $k^{\prime}$. Therefore, this assumption must be false. The proof is complete. 
\end{proof} 

\begin{theorem}
\label{Theorem 3.6} 
With approximate subgradients and PSADLA, the level value converges as
\begin{equation} \label{limit of level}
    \lim _{k \rightarrow \infty} \bar{f}_{k}=f^{\star}.
\end{equation}
Moreover, we have
\begin{equation} \label{limit of f}
    \lim _{k \rightarrow \infty}\left\{\inf _{\kappa \leq k} f\left(x^\kappa\right)\right\}=f^{\star}.
\end{equation}
\end{theorem}

\begin{proof} \textbf{Proof for \eqref{limit of level}:} 
First, we rewrite the level-update formula \eqref{level adjustment formula surrogate subgradient} in terms of $k(j)$ and $k(j+1)$ as
\begin{equation}\label{level adjustment formula surrogate subgradient with kj expression}
    \bar{f}_{k(j+1)}=\frac{\gamma}{\bar{\gamma}} \bar{f}_{k(j)}+\left(1-\frac{\gamma}{\bar{\gamma}}\right) \min _{\kappa \in\{k(j), \dots, k(j+1)-1\}}F^{\kappa}.
\end{equation}
By Lemma \ref{infinitely_offen_surrogate_subgradient}, the level value is adjusted infinitely often as $k \to \infty$. Since the level value remains unchanged unless the PSVD problem is infeasible, proving \eqref{limit of level} is equivalent to showing that $\lim_{j \to \infty} \bar{f}_{k(j)} = f^{\star}$. 
From Theorem \ref{Theorem 3.1}, we know that the sequence $\{\bar{f}_{k(j)}\}_{j=1}^{\infty}$ is monotonically increasing and is bounded above by $f^{\star}$. Thus, it must have a finite limit. We proceed by contradiction to show that this limit is $f^{\star}$.
Assume the limit is strictly less than $f^{\star}$. Then there exists a sufficiently large $k^{\prime}$ and a positive scalar $\bar{\beta}$ such that $f^{\star}-\bar{f}_{k(j)} \geq \bar{\beta}$ holds for any $j$ that $k(j) \geq k'$. Together with Condition \ref{condition on Fk}, this implies that the following inequality holds for any $j$ that $k(j) \geq k'$:
\begin{equation}
     \min _{\kappa \in\left\{k(j), k(j)+1, \ldots, k(j+1)-1\right\}}F^\kappa -\bar{f}_{k(j)} \geq \min(\varepsilon,\bar{\beta}). 
\end{equation}
Following a similar line of reasoning as in the proof of Theorem \ref{Theorem 2.4}, we find that for  any $j$ that $ k(j) \geq k^{\prime}$
\begin{equation}
\begin{aligned}
    \bar{f}_{k(j+1)}-\bar{f}_{k(j)}
    &=\left(1-\frac{\gamma}{\bar{\gamma}}\right) \left(\min_{\kappa \in\left\{k(j), k(j)+1, \ldots, k(j+1)-1\right\}} F^\kappa -\bar{f}_{k(j)} \right)\\ 
    &\geq \left(1-\frac{\gamma}{\bar{\gamma}}\right) \min(\varepsilon,\bar{\beta}).
\end{aligned}
\end{equation}
Summing these inequalities over all $j$ with $k(j) \geq k^{\prime}$, we get
\begin{equation}
    \sum_{j: k(j) \geq k^{\prime}}^{+\infty}\left(\bar{f}_{k(j+1)}-\bar{f}_{k(j)}\right) 
    = +\infty,
\end{equation}
which implies that $\{\bar{f}_{k(j)}\}_{j=1}^{\infty}$ is unbounded. This contradicts the earlier conclusion that it is bounded by $f^{\star}$. Hence, our assumption must be false, and we conclude that
$\bar{f}_{k(j)} \rightarrow f^{\star}$ as $j \rightarrow \infty$. This proves \eqref{limit of level}.

\textbf{Proof for \eqref{limit of f}:} 
To prove the convergence stated in \eqref{limit of f}, we proceed similarly to the proof of Theorem \ref{Theorem 2.4}. Letting $j \to \infty$ in \eqref{level adjustment formula surrogate subgradient}, we obtain
\begin{equation}
    \lim _{j \rightarrow \infty} \left\{\min _{\kappa \in\left\{k(j), k(j)+1, \ldots, k(j+1)-1\right\}}F^{\kappa} \right\}=f^{\star}.
\end{equation}
To simplify the expression, we denote the iteration at which the minimum is attained in the above formula as $k^{\prime}(j)$, i.e., $k^{\prime}(j)=\arg\min_{\kappa \in\{k(j), \ldots, k(j+1)-1\}} F^{\kappa}$. By Condition \ref{condition on Fk}, for each $j$ either $F^{k^{\prime}(j)} = f(x^{k^{\prime}(j)})$ or $\bar{f}_{k^{\prime}(j)}+\varepsilon \leq F^{k^{\prime}(j)}<f(x^{k^{\prime}(j)})$. Since both $\bar{f}_{k^{\prime}(j)}$ and $F^{k^{\prime}(j)}$ converge to $f^{\star}$, for sufficiently large $j$, the latter inequality cannot hold. We thus have 
\begin{equation}
    \lim _{j \rightarrow \infty} f\left(x^{k^{\prime}(j)}\right) =f^{\star}.
\end{equation}
Because $\{f(x^{k^{\prime}(j)})\}_{j=1}^{\infty}$ is a subsequence of $\{f(x^k)\}_{k=1}^{\infty}$ and $f(x^k)$ is bounded below by $f^{\star}$, it follows that
\begin{equation}
    \lim _{k \rightarrow \infty}\left\{\inf _{\kappa \leq k} f\left(x^\kappa\right)\right\}=f^{\star}. 
\end{equation}
This completes the proof.
\end{proof}

\subsection{Applications}
\label{subsection 3.2}

Having established convergence in the prior subsection, the attention of this subsection is on the practical considerations: how to calculate the approximate subgradient that satisfies \eqref{approximate subgradient definition} for a given problem while satisfying  
Condition \ref{condition on Fk}. Generic convex problems with additive objective functions are considered first. Then, the dual problems arising within Lagrangian Relaxation (LR) for Integer Programming (IP) are considered.

\subsubsection{Convex Problems with Additive Objective Functions}
In this subsection, we consider the following convex optimization problem with an additive objective function:
\begin{equation}\label{Convex_optimization_with_additive_objective}
\begin{aligned}
& \underset{x\in \mathbb{R}^n}{\min}
& & \sum_{i=1}^{I}f_{i}(x), & \text{s.t.}
& & x \in \mathcal{X},
\end{aligned}
\end{equation}
where $f_i$ is the $i$-th component of the objective function. For simplicity, we denote the entire objective function as $f(x) = \sum_{i=1}^{I}f_{i}(x)$. In the following, it is shown that calculating a subgradient for one or a few components is sufficient to obtain an approximate subgradient.
\begin{proposition}
Suppose $g_{i,\tau_i}$ is the subgradient of the $i$-th component obtained at iteration $\tau_i$ ($< k$) at point $x^{\tau_i}$. Given $x^{k}$ at iteration $k$, we calculate the subgradient for a subset of objective components $\Xi \subset \{1, \dots, I\}$. Then,
\begin{equation} \label{p3.2_1}\tilde{g}^{k}=\sum_{j\in\Xi}g_{j,k}+\sum_{i\notin\Xi} g_{i,\tau_i},
\end{equation}
and  
\begin{equation}
\label{surrogate_obj_value_defination}
   F^k=\sum_{j\in\Xi}f_j(x^{k})
    +\sum_{i\notin\Xi} \left(f_i(x^{\tau_i})-g_{i,\tau_i}^{\top}\left(x^{\tau_i}-x^{k}\right)\right)
\end{equation}
satisfy inequality \eqref{approximate subgradient definition}. If $F^k$ satisfies $\bar{f}_{k} < F^k \leq f(x^k)$, then $\tilde{g}^{k}$ is an approximate subgradient of $f$ at $x^k$.
\end{proposition}

\begin{proof}
Multiplying \eqref{p3.2_1} by $\left(x^{\star}-x^{k}\right)$ gives:
    \begin{equation}\begin{aligned}
    & \left(\tilde{g}^{k}\right)^{\top}\left(x^{\star}-x^{k}\right)
    =\left(\sum_{j\in\Xi}g_{j,k}\right)^{\top}\left(x^{\star}-x^{k}\right)+\left(\sum_{i\notin\Xi} g_{i,\tau_i}\right)^{\top}\left(x^{\star}-x^{k}\right)\\
    &=\sum_{j\in\Xi}\left(g_{j,k}\right)^{\top}\!\left(x^{\star}\!-\!x^{k}\right)\!+\!\sum_{i\notin\Xi} \left(g_{i,\tau_i}^{\top}\!\left(x^{\star}\!-\!x^{\tau_i}\right)+g_{i,\tau_i}^{\top}\left(x^{\tau_i}-x^{k}\right)\right), \forall x^{\star} \in \mathcal{X}^{\star}.
\end{aligned}\end{equation}
Since $g_{i,\tau_i}$ is a subgradient of the $i$-th component at $x^{\tau_i}$, we have 
\begin{equation}
\label{subgradient_componnent_subgradient_property}
    f_i(x) - f_i(x^{\tau_i}) \geq \big(g_{i,\tau_i}\big)^{\top}(x - x^{\tau_i}), \quad \forall x \in \mathcal{X},
\end{equation}
and the following inequality thus holds
\begin{equation}\begin{aligned}
    \left(\tilde{g}^{k}\right)^{\top}\!\!\left(x^{\star}\!-\!x^{k}\right) \!\leq\! \sum_{j\in\Xi}\!(f_j(x^{\star})\! -\! f_j(x^{k}))
    \!+\!\!\sum_{i\notin\Xi} \!\left(f_i(x^{\star}) \!-\! f_i(x^{\tau_i})\!+\! g_{i,\tau_i}^{\top}\!\!\left(x^{\tau_i}\!-\!x^{k}\right)\right).
\end{aligned}\end{equation}
Rearranging terms yields
\begin{equation}\begin{aligned}
    f^{\star} - F^{k} \geq \left(\tilde{g}^{k}\right)^{\top}\left(x^{\star}-x^{k}\right), 
\end{aligned}\end{equation}
which is precisely inequality \eqref{approximate subgradient definition}. The proof is complete.
\end{proof}

Condition \ref{condition on Fk} is instrumental for convergence and is satisfied by using the following algorithmic procedure:
Given an initial solution $x^0$, initial level $\bar{f}_0$, and $\varepsilon$ (initialized as a small positive scalar), subgradients $g_{i,0}$ of all the components at $x^0$ are calculated. The first solution $x^1$ is generated by using the subgradient at $x^0$ with a stepsize calculated as \eqref{polyak stepsize with estimation}. At iteration $k\geq 1$, given $x^k$ and $\bar{f}_k$, the subgradient $g_{i,\tau_i}$ components are calculated incrementally until the inequality $\bar{f}_{k}+\varepsilon \leq F^k$ in Condition \ref{condition on Fk} is satisfied. In the worst case, this inequality cannot be satisfied after calculating the subgradient for all the components.  In this case, $F^{k}=f(x^k)$, and Condition \ref{condition on Fk} holds. With the approximate subgradient calculated, 
the solution is then updated as $x^{k+1}=P_\mathcal{X}\left(x^k-s^k \tilde{g}^k\right)$ with the stepsize \eqref{polyak stepsize surrogate subgradient}, and the level value is adjusted based on the PSADLA approach.

The approximate subgradient method presented above offers a promising approach for efficiently solving convex optimization problems with additive objective functions, particularly when these functions involve a large number of summation components. Such problems frequently arise in a variety of practical applications. Specifically, network optimization problems are defined on a network with vertices and edges, and the objective function is sometimes additive in terms of nodes or edges \cite{lobel2010distributed,ram2009incremental,larsson1996conditional}. Moreover, an optimal control problem is generally formulated as a convex optimization problem to optimize a sequence of inputs for a dynamic system with a sequence of stages. Additive objective functions are usually used, as they can represent the cumulative cost or revenue of the system over a certain period \cite{zhang2023}.  Furthermore, a supervised learning approach trains a machine learning model to minimize the average loss on a training set \cite{caruana2006empirical}. The objective function is additive, with each component associated with a piece of data. Polyak stepsizes have recently been integrated with stochastic gradient descent approaches, showing promising results in solving such problems \cite{orvieto2022dynamics}. In our approximate subgradient approaches, to guarantee convergence, calculating the subgradient of one or a few components is enough, enabling the use of a mini-batch of data in each iteration. 

\subsubsection{Lagrangian Duals  
of Integer Programming Problems}
\label{subsection: application for dual problems}

In this subsection, we study the Lagrangian duals  of  Integer Programming (IP) problems, which have extensive applications in practical operational optimization. We first present a generic formulation of IP and then propose an approach for computing approximate subgradients while ensuring the satisfaction of Condition \ref{condition on Fk}.

A generic IP problem involves integer decision variables $y\in \mathcal{Y} \subseteq \mathbb{Z}^d$, and can be formulated as
\begin{equation}
\begin{aligned}
\underset{y\in \mathcal{Y}}{\max}
\quad & J(y), & \text{s.t.} \quad 
& g(y) \leq 0,
\label{ILP}
\end{aligned}
\end{equation}
where $J: \mathbb{Z}^d \mapsto \mathbb{R}$ is the objective function and $g: \mathbb{Z}^d \mapsto \mathbb{R}^n$ is the constraint function. Within the Lagrangian Relaxation (LR) framework, the constraints are relaxed by multipliers $x\in \mathbb{R}^{n}$, leading to the relaxed problem: 
\begin{equation}
\begin{aligned}
\underset{y\in \mathcal{Y}}{\max}
\quad  L\left(y,x\right) = J(y) + x^{\top} g\left(y\right), 
\label{relaxed problem}
\end{aligned}
\end{equation}
where $L$ is the Lagrangian function.  For any given set of multipliers, the optimal value of the relaxed problem \eqref{relaxed problem} serves as an upper bound on the optimal value of the original IP problem \eqref{ILP}. The goal of LR is to optimize the multipliers to find the ``best” upper bound, i.e., to minimize the following dual problem:
\begin{equation}\begin{aligned} \label{dual_problem_ILP}
&\underset{x\in\mathbb{R}^n}{\min}
\quad  f(x) =  \sup_{y\in \mathcal{Y}} L(y,x), \\
&\text{s.t.}
\quad  x \geq 0.
\end{aligned}\end{equation}
Since $\mathcal{Y}$ is a discrete set containing a finite number of points, the dual function $f$ is a piece-wise linear, non-smooth, convex function, with each linear segment corresponding to a point in $\mathcal{Y}$. Subgradient methods are thus typically used to optimize \eqref{dual_problem_ILP}. However, to compute a subgradient at $x^k$,  one needs to solve the relaxed problem \eqref{relaxed problem} to optimality, which can be computationally expensive. Below, we show that with our approximate subgradient, it suffices to obtain only a ``good enough" solution to the relaxed problem, thus reducing computational effort.

\begin{proposition}
\label{Prop_addtive_appro_subgradient}
For a given $x^k$, any feasible solution $\tilde{y}^k$ of the relaxed problem \eqref{relaxed problem} satisfies the following inequality:
\begin{equation}\label{approximation_subgradient_inequality}
f^{\star}-L(\tilde{y}^k, x^k) \geq g(\tilde{y}^k)^{\top}(x^{\star}-x^k), \quad\forall x^{\star}\in \mathcal{X}^{\star}.
\end{equation}
If $\tilde{y}^k$ yields a sufficiently high objective value such that 
\begin{equation} \label{dual_of_ip_optimization_condition}
    L(\tilde{y}^k, x^k) > \bar{f}^k,
\end{equation}
then $g(\tilde{y}^k)$ is an approximate subgradient of $f$ under our definition with $F^k=L(\tilde{y}^k, x^k)$ and  $\tilde{g}^k=g(\tilde{y}^k)$.
\end{proposition}

\begin{proof}
The satisfaction of \eqref{approximation_subgradient_inequality} can be derived as follows:
\begin{equation}
\begin{aligned}
    L(\tilde{y}^k, x^k)&=J(\tilde{y}^k)+(x^k)^{\top} g(\tilde{y}^k)\\
    &=J(\tilde{y}^k)+(x^k-x^{\star}+x^{\star})^{\top} g(\tilde{y}^k)\\
    &=L(\tilde{y}^k, x^{\star}) + (x^k-x^{\star})^{\top} g(\tilde{y}^k)\\
    &\leq f^{\star} + (x^k-x^{\star})^{\top} g(\tilde{y}^k).\\
\end{aligned}
\end{equation}
By the definition of approximate subgradients, 
if $L(\tilde{y}^k, x^k) > \bar{f}^k$, then $g(\tilde{y}^k)$ acts as an approximate subgradient with $F^k=L(\tilde{y}^k, x^k)$. The proof is complete.
\end{proof}

Condition \ref{condition on Fk} can be enforced through the following algorithmic procedure to ensure convergence.  At each iteration $k$, for given $x^k$, level value $\bar{f}_k$, and the parameter $\varepsilon$ in Condition \ref{condition on Fk}, we solve the relaxed problem. If an objective value that is no less than $\bar{f}_k + \varepsilon$ is obtained, then the condition is satisfied and $x^k$ is updated as
$x^{k+1}=[x^{k}-s^{k}g(\tilde{y}^k)]^{+}$, where $s^k$ is calculated according to \eqref{polyak stepsize surrogate subgradient}. 
In the worst case, even if the relaxed problem is optimally solved, its objective value may not reach $\bar{f}_k+\varepsilon$. In this case, $F^k=f(x^k)$ and $x^k$ is updated based on the optimal relaxed problem solution as $x^{k+1}=[x^{k}-s^{k}g(\tilde{y}^{\star})]^{+}$.

Our approximate subgradient method offers a direction for efficiently solving practical operation optimization problems with integer decision variables. Most importantly, it can exploit the separable structure often present in practical operational optimization problems. In the following, we will detail this aspect. 
To ensure the balance of supply and demand while minimizing costs, power system unit commitment problems are solved daily to determine which generating units should be on or off and to decide their production levels. In manufacturing systems, scheduling problems are solved before each shift to determine the processing sequence of jobs, aiming to achieve objectives such as on-time deliveries \cite{pinedo2012scheduling}. These problems are generally separable \cite{nikolaidis2020enhanced, hou2023fast,LIUmlSLR}.
After relaxing the coupling constraints, the problem can be decomposed into a set of subproblems; that is, the relaxed problem takes the following form:
\begin{equation}
\begin{aligned}
&\max \quad  \sum_{i=1}^{I}L_{i}\left(y_i,x\right),\\  &\text{s.t.} \quad y_i\in \mathcal{Y}_i, \quad \forall i\in\{1,2,\dots,I\},
\end{aligned}
\end{equation}
where $L_{i}$ is the objective function of subproblem $i$, $y_i$ denotes the decision variables associated, and $\mathcal{Y}_i$ denotes the feasible set. As presented above, to calculate an approximate subgradient at a point, it is sufficient to obtain a ``good enough'' solution. For these problems with separable structures, only one or several subproblems need to be solved. Moreover, the optimal subproblem solution is not required, and exact algorithms such as Dynamic Programming or Branch-and-Cut are thus unnecessary. Instead, inexact yet computationally efficient methods, such as machine-learning-based methods \cite{LIUmlSLR} or heuristics \cite{liu2020ordinal}, can be employed to obtain  ``good enough" subproblem solutions. Furthermore, these problems generally need to be solved daily or hourly. Our approach is level-based. The level value from one solve can be used to warm start the next, leading to a high computational efficiency.

In the rest of this subsection, we compare our approach with the surrogate subgradient approach \cite{zhao1999surrogate, bragin2015convergence,bragin2022surrogate}, which is an efficient and popular approach for solving the dual problems \eqref{dual_problem_ILP} of various practical IP problems. We first summarize the key idea of the surrogate subgradient approach, and then we demonstrate that the surrogate subgradient essentially falls within the definition of our approximate subgradient.
In each iteration of the approach, the relaxed problem \eqref{relaxed problem} is approximately solved to satisfy the following Surrogate Optimality Condition (SOC) \cite[Eq. (28)]{zhao1999surrogate}:
\begin{equation} \label{Sorrogate_Optimality_Condition}
    L\left(\tilde{y}^k, x^k\right)>L\left(\tilde{y}^{k-1}, x^k\right).
\end{equation}
When a solution $\tilde{y}^{k}$ satisfying \eqref{Sorrogate_Optimality_Condition} is found, $g( \tilde{y}^k)$ becomes a ``surrogate subgradient” of $f$ at $x^k$ and is used to update $x^k$. If $\tilde{y}^{k-1}$ happens to be an optimum of \eqref{relaxed problem} for $x^k$, a solution $\tilde{y}^{k}$ that satisfies \eqref{Sorrogate_Optimality_Condition} does not exist. In this case, $\tilde{y}^k$ is set as $\tilde{y}^{k-1}$. The value of $L\left(\tilde{y}^k, x^k\right)$ is less than or equal to the corresponding dual value $f(x^k)$, and is referred to as the ``surrogate" dual value. The solution is then updated as $x^{k+1}=[x^k-s^k g( \tilde{y}^k)]^+$. In \cite{bragin2022surrogate}, the stepsize $s^k$ is calculated based on the level value as $s^k=\gamma \cdot \frac{L(\tilde{y}^k, x^k)-\bar{f}_{k}}{\left\|g( \tilde{y}^k)\right\|^2}, 0<\gamma<1$. Note here that $\gamma$ is less than 1.

The surrogate subgradient satisfies the definition of our approximate subgradient. This holds because, with the SOC \eqref{Sorrogate_Optimality_Condition} and the stepsize defined above, $L(\tilde{y}^k, x^k)$ is always greater than the level value. We formalize this result in the following proposition.

\begin{proposition}
Suppose that all the level values $\bar{f}_k$ for $k \in \{0,\dots,K\}$ are identical.
At iteration $k=0$, we start with an initial solution $x^0$ and $\tilde{y}^0$, which satisfies $L(\tilde{y}^0, x^0)>\bar{f}_0$. 
Using the surrogate subgradient and stepsize as described above, for any iteration $k > 0$, the surrogate dual value $L(\tilde{y}^k, x^k)$ remains greater than the level value $\bar{f}_k$. Moreover, the surrogate subgradient $g(\tilde{y}^k)$ satisfies 
\begin{equation}
f^{\star}-L(\tilde{y}^k, x^k) \geq g(\tilde{y}^k)^{\top}(x^{\star}-x^k), \quad\forall x^{\star}\in \mathcal{X}^{\star}.
\end{equation}
\end{proposition}
\begin{proof}
    By the definition of the Lagrangian function $L$, we have 
\begin{equation}\begin{aligned}
       L\left(\tilde{y}^{k}, x^{k+1}\right)=&J(\tilde{y}^{k}) + \left(x^{k+1}\right)^{\top}  g(\tilde{y}^{k})\\
       =&L\left(\tilde{y}^{k}, x^{k}\right) + \left(x^{k+1}-x^{k}\right)^{\top}  g(\tilde{y}^{k})\\
       =&L\left(\tilde{y}^{k}, x^{k}\right) + \left([x^{k}-s^{k} g( \tilde{y}^{k})]^+-x^{k}\right)^{\top} g(\tilde{y}^{k})
\end{aligned}\end{equation}
As shown in Lemma 2.1 of \cite{sun2007surrogate}, the term $[x^{k}-s^{k} g( \tilde{y}^{k})]^+-x^{k}$ must be greater than or equal to $-s^{k} g( \tilde{y}^{k})$. Therefore, 
\begin{equation}
       L\left(\tilde{y}^{k}, x^{k+1}\right)\geq L\left(\tilde{y}^{k}, x^{k}\right) - s^{k}  \|g(\tilde{y}^{k})\|^2.
\end{equation}
Substituting  $s^{k}=\gamma \cdot \frac{L(\tilde{y}^{k}, x^{k})-\bar{f}_{k}}{\left\|g( \tilde{y}^{k})\right\|^2}$, we get 
\begin{equation}\begin{aligned}
L\left(\tilde{y}^{k}, x^{k+1}\right) \geq &L\left(\tilde{y}^{k}, x^{k}\right) - \gamma\left( L(\tilde{y}^{k}, x^{k})-\bar{f}_{k}\right)\\
=&(1-\gamma) \cdot L\left(\tilde{y}^{k}, x^{k}\right) + \gamma\bar{f}_{k}.
\end{aligned}\end{equation}
Due to the satisfaction of the SOC \eqref{Sorrogate_Optimality_Condition}, we have 
\begin{equation}
L\left(\tilde{y}^{k+1}, x^{k+1}\right) >
(1-\gamma) \cdot L\left(\tilde{y}^{k}, x^{k}\right) + \gamma\bar{f}_{k}.
\end{equation}
Since $ 0<\gamma<1$, the right-hand side represents a convex combination. If $L(\tilde{y}^{k}, x^{k})>\bar{f}_{k}$, then $L(\tilde{y}^{k+1}, x^{k+1})>\bar{f}_{k}$. Since this is true for $k=0$, it follows inductively that $L(\tilde{y}^{k}, x^{k}) > \bar{f}_k$ for all $k \in \{0,\dots,K\}$. The satisfaction of inequality \eqref{approximation_subgradient_inequality} can be established in the same way as in the proof of Proposition \ref{Prop_addtive_appro_subgradient}. The proof is complete. 
\end{proof}

\section{Numerical Results} 
\label{section 4}
In this section, multiple convex optimization problems of various sizes and characteristics are tested. Across all the instances tested, $\Bar{\gamma}$ is set to 1 and $\gamma$ is set to 0.5, without fine-tuning. The feasibility of a PSVD problem is checked using the default LP solver provided by IBM CPLEX. The testing platform is a laptop with an Intel i7-13700H, 32GB of RAM, Windows 11, MATLAB 2023b, and CPLEX 12.10. The datasets and the code implementing our approach are available upon request. 

\subsection{$L_1$ Approximation Problems}
\label{subsection 4.1}

In this subsection, we consider the classical $L_1$ approximation problem (see Chapter 6.1 of \cite{boyd2004convex}), which can be formulated as:

\begin{equation}\begin{aligned}\label{L1_norm}
& \underset{x \in \mathbb{R}^N }{\min}
& & \|A x-b\|_1:=\sum_{m=1}^{M}|A_m x-b_m|, \\
\end{aligned}\end{equation}
where $A \in \mathbb{R}^{M \times N}$, $b \in \mathbb{R}^{M}$, $A_m$ denotes the $m$-th row of $A$, and $b_m$ denotes the $m$-th element of $b$. One can readily derive that a subgradient at $x^k$ is 
\begin{equation}g^k=A^{\top} z^k,\end{equation}
where $z^k$ is calculated as
\begin{equation}
    z_m^k= \begin{cases}\operatorname{sign}\left(A_m x^k-b_m\right) & \text { if } A_m x^k-b_m \neq 0; \\ 0 & \text { if } A_m x^k-b_m=0.\end{cases}
\end{equation}

We consider an instance with $M=500$ and $N=100$, where the elements of $A$ are randomly generated using a continuous uniform distribution $\mathcal{U}[-1,1]$. For convenience, we set $b$ to be a zero vector, so that the optimal solution is the zero vector and the optimal value is zero. 
We first solve the instance using the subgradient method with PSADLA described in Section \ref{section 2}. The initial solution is randomly sampled from the uniform distribution $\mathcal{U}[-10, 10]$, and the initial level value is set to $-1000$. These initial conditions are chosen to be far from the optimal solution. The results are shown in red in Fig. \ref{L1_sub1}. After 103 iterations, the level value converges to within 10 units of the optimal value. By iteration 90, a solution lies within an $L_2$ distance of at most 0.01 from the optimal solution. We also compare the new method with the SDD-based level adjustment approach \cite{bragin2022surrogate}; the corresponding results are shown in blue in Fig.~\ref{L1_sub1}. The PSVD-based PSADLA method proposed in this paper exhibits consistently superior performance compared to the SDD-based approach.

\begin{figure}[h!]
    \centering
    \begin{subfigure}[b]{0.7\textwidth}
        \includegraphics[width=\textwidth]{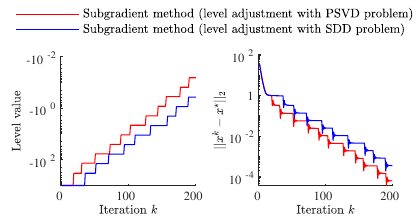}
        \caption{Comparison between PSVD-based and SDD-based level adjustment methods}
        \label{L1_sub1}
    \end{subfigure}
    \hfill
    \begin{subfigure}[b]{0.7\textwidth}
        \includegraphics[width=\textwidth]{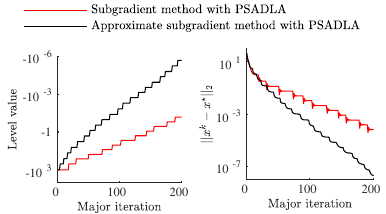}
        \caption{Comparison between subgradient and approximate subgradient methods}
        \label{L1_sub2}
    \end{subfigure}
    \caption{Results of solving the $L_1$ approximation problem}
    \label{L1_norm_surrogate_subgradient_results}
\end{figure}

The approximate subgradient method introduced in Section \ref{section 3} is also tested using the same instance as above, with a total of $M=500$ objective components. We set $\varepsilon$ in Condition \ref{condition on Fk} to  $10^{-10}$. Every 50 components are grouped, and the subgradient of a group is calculated per iteration to obtain an approximate subgradient. The numerical results show that Condition \ref{condition on Fk} is always satisfied after computing the subgradients for each group. Hence, after computing the subgradient for all the components once, the solution undergoes ten updates. Upon termination, the algorithm obtains a level value within $10^{-6}$ units of the optimal value and a solution within an $L_2$ distance of no greater than $2 \times 10^{-8}$ from the optimal solution. For fair comparison with the subgradient method, we define a ``major iteration" as one in which the subgradients of all objective components are calculated exactly once. The results are shown in black in Fig. \ref{L1_sub2}. As illustrated, the approximate subgradient method requires fewer major iterations to obtain solutions of comparable quality.

\subsection{Dual Problems of Generalized Assignment Problems}\label{subsection 4.2.1}

In this example, we consider the dual problem of the generalized assignment problem \cite{cattrysse1992survey}. The generalized assignment problem involves assigning a set of jobs $\mathcal{J}$ to a set of machines $\mathcal{M}$ to minimize $\sum_{j\in\mathcal{J}}\sum_{m\in\mathcal{M}} c_{j, m} y_{j,m}$, subject to assignment constraints $\sum_{m\in\mathcal{M}} y_{j, m}=1, \forall j\in\mathcal{J}$ and capacity constraints $\sum_{j\in\mathcal{J}} t_{j, m} y_{j, m} \leq T_m, \forall m\in\mathcal{M}$, where $\{y_{j,m}\}$ are binary decision variables, $c_{j,m}$ is the cost of assigning job $j$ to machine $m$, $t_{j,m}$ is the time required to process job $j$ on machine $m$, and $T_m$ is the capacity of machine $m$. The dual problem to be optimized is
\begin{equation}\begin{aligned}
& \underset{x \in \mathbb{R}^n}{\max}
& & f(x), 
& \text{s.t.}
& & x \geq 0,
\label{GAP_dual_problem_relax_capacity}\end{aligned}\end{equation}
where $f(x)=\inf _{y \in \mathcal{Y}_1}L(x,y)$ is the dual function with the Lagrangian function
\begin{equation}
    L(x,y)=\sum_{j\in\mathcal{J}} \sum_{m\in\mathcal{M}} c_{j, m} y_{j, m}+\sum_{m\in\mathcal{M}} x_m\left(\sum_{j\in\mathcal{J}} t_{j, m} y_{j, m}-T_m\right)
\end{equation} 
and $\mathcal{Y}_1=\left\{y\in\{0,1\}^{|\mathcal{J}|\times|\mathcal{M}|} : \sum_{m\in\mathcal{M}} y_{j, m}=1, \forall j \in \mathcal{J}\right\}$. The dual function $f$ is a piece-wise linear non-smooth function. To compute a subgradient at a given point $x^k$, $L(x^k,y)$ is exactly optimized subject to $\mathcal{Y}_1$. 
Since the dual problem \eqref{GAP_dual_problem_relax_capacity} is a maximization problem, the following modifications to our method are required: the iterate becomes $x^{k+1}=[x^k + s^k g^k]^+$, the level value becomes an over-estimate of the optimal dual value, and the stepsize is calculated as $s^k=\gamma \cdot (\bar{f}_k-f(x^k))/\|g^k\|^2$. 

\begin{figure}[t]
    \begin{center}
        \includegraphics[width=13.0 cm]{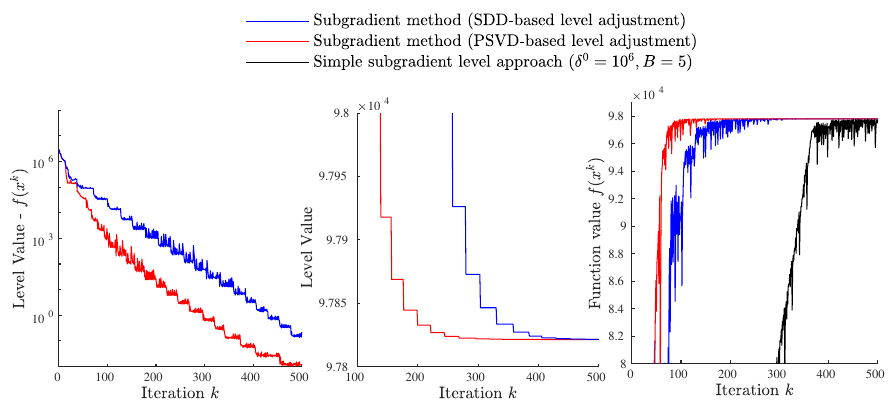}
    \end{center}
    \caption{Results of solving the dual of GAP d201600: level values and objective function values across iterations}
    \label{results of d201600}
\end{figure} 

We examine three standard test instances, d201600, d401600, and d801600, from the OR Library \cite{beasley1990or}. The first instance tested is d201600, which consists of 1600 jobs and 20 machines. 
We first solve the instance by using the subgradient method with our PSVD-based level adjustment approach. To test the robustness of the proposed method, the initial conditions are intentionally set far from optimality. Specifically, the initial level value is set to 500,000, and the initial solution $x^0$ is randomly sampled from the uniform distribution $\mathcal{U}[0, 100]$. The algorithm terminates when the total number of iterations reaches 500. In our approach, the level values are upper bounds to the optimal dual value, whereas the dual values are lower bounds, and the difference can be used as a measure of solution quality. As shown in the left sub-figure of Fig. \ref{results of d201600}, this difference decreases to $10^{-2}$, demonstrating the achievement of high-quality solutions. To provide a clearer visualization, the level values and objective function values $f(x^k)$ are also plotted in Fig. \ref{results of d201600}. For comparison, we also test the level adjustment approach based on the SDD problem of \cite{bragin2022surrogate}. The results obtained using the SDD-based approach are inferior to those of the proposed method.

\input{comparison_subtable_d201600_d401600}

\input{comparison_subtable_d801600}

To more clearly demonstrate the advantages of our level adjustment approach, in the following, we will define a metric to empirically compare the convergence rate of our approach and several commonly used  stepsize rules for subgradient methods. Specifically, we consider the metric that the number of iterations required for an approach to obtain a solution whose objective function value is within 1\%, 0.5\%, and 0.1\% of the optimal value for the first time. Using the metric, we compare our approach  with the following approaches:
\begin{itemize}
    \item Subgradient method with the Polyak stepsize and the SDD-based level adjustment \cite{bragin2022surrogate};
    \item Subgradient method with the Polyak stepsize and the Path-based level adjustment \cite{goffin1977convergence};
    \item Subgradient method with the diminishing stepsize $s_k=a/\sqrt{k}, a>0$;
    \item Subgradient method with the square summable but not summable stepsize $s_k=a/(k+b), a>0, b\geq0$.
\end{itemize}

The known optimal values for instances d201600, d401600, and d801600 are 97,821.35, 97,105, and 97,034, respectively. To evaluate performance, we test two initial solutions: a zero vector and a vector where all elements are set to 100. The results corresponding to these initial solutions are presented under the columns labeled ``$x^0: 0$" and ``$x^0: 100$" in Table \ref{Comparison_of_rules_d201600_d401600} and Table \ref{Comparison_of_rules_d801600}. 

Our PSVD-based level adjustment approach is first tested with initial level values of $1 \times 10^{5}$, $2 \times 10^{5}$, and $5 \times 10^{5}$. The number of iterations required to obtain solutions of the required quality is listed in Table \ref{Comparison_of_rules_d201600_d401600} and Table \ref{Comparison_of_rules_d801600}. The results demonstrate that high-quality solutions can be obtained efficiently for all three instances, and the performance of the proposed method is not sensitive to the initial settings. The SDD-based level adjustment approach from \cite{bragin2022surrogate} is also tested. As demonstrated in the tables, its performance is inferior due to the tighter formulation of the novel PSVD problem, as established in Proposition \ref{Prop 2.1}. The tighter PSVD formulation allows more frequent level adjustments, leading to faster convergence. 

Additionally, we test the path-based level adjustment method proposed in \cite{goffin1977convergence}. The approach has two hyperparameters: $\delta^0$ and $B$. The selection of $\delta^0$, which influences the level value, is highly dependent on the value of $B$. We test the method under $\delta^0 \in \{5 \times 10^4, 1 \times 10^5, 5 \times 10^5, 1 \times 10^6\}$ and $B \in \{1, 5, 10, 50, 100\}$. The results are shown in Table \ref{Comparison_of_rules_d201600_d401600} and Table \ref{Comparison_of_rules_d801600}. For certain combinations of hyperparameters, the number of iterations required exceeds 1000. In such cases, the algorithm is terminated, and the result is denoted by a dash (``-"). The results indicate that the path-based approach is highly sensitive to the choice of hyperparameters. For instance, with $\delta^0 = 1 \times 10^6$ and $B = 1$, reasonably fast convergence is achieved for instance d201600 when the initial solution is a zero vector. However, many other combinations of hyperparameters and initial settings lead to significantly degraded performance, highlighting the substantial effort required for parameter fine-tuning. 

Furthermore, we test the diminishing stepsize and the square-summable-but-not-summable stepsize. For each stepsize, there are many choices. We use two commonly referenced selections from the lecture notes of Stanford University EE364b\footnote{The lecture note can be accessed via \href{https://see.stanford.edu/materials/lsocoee364b/02-subgrad_method_notes.pdf} {https://see.stanford.edu/materials/lsocoee364b/02-subgrad\_method\_notes.pdf}}.
For the diminishing stepsize, $s_k = a/\sqrt{k} $ with $a \in \{10^{-6}$, $10^{-5}$, $10^{-4}$, $10^{-3}$, $10^{-2}$, $10^{-1}$, $10^0\}$. For the square-summable-but-not-summable stepsize, $s_k = a/(k+b)$ with $a \in \{10^{-6}$, $10^{-5}$, $10^{-4}$, $10^{-3}, 10^{-2}$, $10^{-1}\}$ and $b \in \{0, 10, 100\}$. Both stepsizes are predefined and lack the adaptability inherent in level-based methods, which limits their flexibility during algorithm execution.

\subsection{Dual Problems of Job Assignment Problems with \\Transportation Time}
In this example, a dual problem of another assignment problem is considered, and the approximate subgradient of Subsection~\ref{subsection: application for dual problems} is tested. Before formulating the dual problem to be optimized, the assignment problem is introduced briefly. The problem is to assign a set of jobs $\mathcal{I}$ to a set of machines $\mathcal{M}$, and job $i$ has a set of operations $\mathcal{J}_i=\{1,2, \dots, |\mathcal{J}_i|\}$. If the $j$-th operation of job $i$ is assigned to machine $m$, then $y_{i,j,m}$ equals 1. Otherwise, it is zero. Each operation needs to be assigned to a unique machine, i.e.,
\begin{equation}
\label{GAP_withtransportation_assignmentconstraints}
    \sum_{m \in \Phi_{(i,j)}} y_{i, j, m}=1, \quad\forall i\in\mathcal{I}, j\in\mathcal{J}_i,
\end{equation}
where $\Phi_{(i,j)}$ is the set of machines eligible to process the $j$-th operation of job $i$. Following \cite{deroussi2008simple, zhang2012genetic}, jobs need to go through a sequence of operations on multiple machines, and moving a job from one machine to another incurs transportation costs.
We use non-negative decision variable $z_{i,j}$ to denote the transportation cost of operation $(i, j)$, which satisfies
\begin{equation}
\label{GAP_withtransportation_transcost}
\begin{aligned}
    &c^{\prime}_{i,m_1,m_2}\left(y_{i, j, m_1}+y_{i, j+1, m_2}-1\right) \leq z_{i,j}, \\ 
    &\forall i\in\mathcal{I}, j=1, \ldots, |\mathcal{J}_i|-1,
    m_1 \in \Phi_{(i,j)}, m_2 \in \Phi_{(i,j+1)},
\end{aligned}
\end{equation}
where $c^{\prime}_{i,m_1,m_2}$ denotes the transportation cost required to move job $i$ from machine $m_1$ to $m_2$. Each machine has a limited capacity, i.e.,
\begin{equation}\label{GAP_withtransportation_capacityconstraints}
    \sum_{(i,j) \in \Psi_m} t_{i, j, m} y_{i, j, m} \leq T_m, \quad\forall m\in\mathcal{M},
\end{equation}
where $\Psi_m$ denotes the set of operations that can be processed by machine $m$, $t_{i, j, m}$ denotes the capacity required by machine $m$ to process ($i, j$), and $T_m$ denotes the capacity of machine $m$.
The objective is to minimize assignment and transportation costs:  
\begin{equation}
    \underset{y, z}{\min}\quad \sum_{i\in\mathcal{I}}\sum_{j\in\mathcal{J}_i} \sum_{m\in \Phi_{(i,j)}} c_{i, j, m} y_{i, j, m} + \sum_{i\in\mathcal{I}}\sum_{j=1}^{|\mathcal{J}_i|-1}z_{i,j},
\end{equation}
where $c_{i, j, m}$ is the cost of processing ($i, j$) on machine $m$. After relaxing the capacity constraints \eqref{GAP_withtransportation_capacityconstraints} by using Lagrangian multipliers $x\in\mathbb{R}^{|\mathcal{M}|}$, the relaxed problem can be written as 
\begin{equation}\label{gap_withtranspose_relaxed_problem}\begin{aligned}
& \underset{y, z}{\min}  \quad
L(y, z, x),  \quad \text{s.t.}  \quad
\eqref{GAP_withtransportation_assignmentconstraints}, \eqref{GAP_withtransportation_transcost},
\end{aligned}\end{equation}
with 
\begin{equation}\begin{aligned}
    L(y, z, x)=\sum_{i\in\mathcal{I}}\sum_{j\in\mathcal{J}_i} \sum_{m\in\mathcal{M}} c_{i, j, m} y_{i, j, m} + \sum_{i\in\mathcal{I}}\sum_{j=1}^{|\mathcal{J}_i|-1}z_{i,j}\\
    +\sum_{m\in\mathcal{M}} x_m\left(\sum_{(i,j) \in \Psi_m} t_{i, j, m} y_{i, j, m} - T_m\right).
\end{aligned}\end{equation}
The dual problem to be optimized in this example can then be established as
\begin{equation}
\begin{aligned}
& \underset{x\in\mathbb{R}^{|\mathcal{M}|}}{\max}
& & q(x)=\inf_{y,z \in \mathcal{Y}_2} L(y, z, x), \quad 
\text{s.t.}
& & x \geq 0,
\label{GAP_withtransportation_dual_problem}
\end{aligned}
\end{equation}
where $\mathcal{Y}_2$ is the feasible region defined by constraints \eqref{GAP_withtransportation_assignmentconstraints} and \eqref{GAP_withtransportation_transcost}.
We consider an instance with 40 machines and 40 jobs, with each job requiring 10 operations. The parameters are randomly generated from uniform distributions. For simplicity, we assume that each machine is eligible to process all operations. The assignment costs $\{c_{i,j,m}\}$, processing times $\{t_{i,j,m}\}$ required by operations, machine capacities $\{T_{m}\}$, and transportation costs $\{c^{\prime}_{i,m_1,m_2}\}$ are sampled from uniform distributions. Since the relaxed problem is an Integer Linear Programming problem, Branch-and-Cut implemented by IBM CPLEX is used to solve it so as to calculate subgradients and approximate subgradients. 
\begin{figure}[htb]
    \begin{center}
        \includegraphics[width=13 cm]{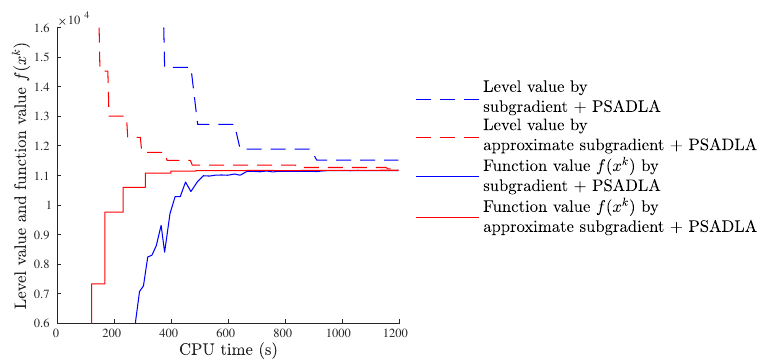}
    \end{center}
    \caption{Results of solving the dual problem \eqref{GAP_withtransportation_dual_problem}: level values and objective function values versus CPU time}
    \label{GAP_TRANS_subgradient_vs_surroagtesubgradient}
\end{figure}

The Lagrangian function in \eqref{gap_withtranspose_relaxed_problem} is additive with respect to jobs, allowing every 10 jobs to be grouped to form a subproblem. The initial level value is set to 20,000, significantly greater than $f^{\star}$. The initial decision variable $x^0$ is sampled from the uniform distribution $\mathcal{U}[0, 100]$. The stopping criterion is set to 1200 seconds, including the time for initialization, subgradient and solution computations, and level adjustment. The results are presented in Fig. \ref{GAP_TRANS_subgradient_vs_surroagtesubgradient}.
Before reaching the stopping criterion, the subproblem solutions are updated 305 times, with an average computation time of 3.82 seconds required to obtain approximate subgradients. Due to the linearity of the PSVD problem in \eqref{PSVD for surrogate subgradient}, feasibility checks are computationally efficient, requiring only 0.06 seconds on average. As a result, high-quality solutions are obtained in a computationally efficient manner. 
For comparison, the subgradient method with PSADLA is also tested, and the results are included in Fig. \ref{GAP_TRANS_subgradient_vs_surroagtesubgradient}. Since solving the relaxed problem \eqref{gap_withtranspose_relaxed_problem} to optimality is necessary to compute a subgradient, the solution is updated only 78 times before reaching the stopping criterion. Consequently, the subgradient approach is slower than the approximate subgradient approach due to the higher computational cost per iteration.

\section{Conclusions}\label{section 5}
In this paper, we develop a Polyak Stepsize with Accelerated Decision-Guided Level Adjustment (PSADLA) for subgradient optimization and provide a rigorous proof of the algorithm's convergence. Testing results across various convex problems with distinct characteristics show that the level value is efficiently adjusted, leading to fast convergence. Furthermore, we show that with PSADLA, obtaining approximate subgradients at each iteration leads to significant CPU time savings while preserving fast convergence. In the following, we discuss the potential limitations of our approach and highlight promising directions for future research.

In the following, we summarize the limitations of our approach. One limitation of our approach lies in the need for periodic feasibility checks of the PSVD problem. In high-dimensional settings, such as those arising in deep neural network training, storing gradients can be memory-intensive, and feasibility checks of the PSVD problem may become computationally expensive. Improving the efficiency of the PSVD feasibility check for extremely large-scale problems (e.g., deep neural network optimization with billions of parameters) is an important avenue for future research. Another challenge occurs in ill-conditioned problems, such as the dual problems of some IP problems. In such cases, the subgradient directions may become almost perpendicular to the direction toward optimal solutions, leading to zigzagging and slow convergence. Integrating our approach with heavy-ball techniques may help mitigate this issue. Finally, machine learning techniques offer a promising avenue for further development. Learning-based predictions of optimal solutions and generating additional constraints to tighten the PSVD problem could enable more frequent level updates, accelerating convergence.

We outline several potential directions for future research in the following. Non-convex problems have significant practical applications and have recently garnered considerable attention. Algorithms for solving non-convex optimization problems often rely on establishing and solving a series of convex approximation problems \cite{razaviyayn2014parallel}. A promising research direction involves integrating our approach with the Successive Convex Approximation (SCA) framework \cite{razaviyayn2014parallel} to solve these convex approximations efficiently. Specifically, the level value obtained from solving one convex approximation could be utilized to warm-start the solution process for subsequent approximations. Moreover, extending our PSVD-based level adjustment method to stochastic gradient settings represents another promising avenue for further exploration. Such extensions could broaden the approach's applicability to large-scale non-convex problems where stochastic methods are commonly employed.

\end{document}

%% file: comparison_subtable_d201600_d401600.tex
\begin{table}[htbp]
    \centering 
    \caption{Results of solving the dual of GAP d201600 and d401600: comparison between multiple stepsize rules} 
    \label{Comparison_of_rules_d201600_d401600}
    {\scriptsize
    \begin{tabular}{ccccccccc}  
        \toprule 
        Method& Param & \multicolumn{2}{c}{d201600}& \multicolumn{2}{c}{d401600} \\
         \cmidrule(lr){3-4} \cmidrule(lr){5-6} 
         &  &  $x^0: 0$& $x^0: 100$& $x^0: 0$&$x^0: 100$\\
        \midrule 
         \multirow{3}{*}{\makecell{Ours}}
        &1e5&  12/36/59&32/44/73&16/79/179&66/123/220\\
        &2e5& 61/78/109&53/76/110&99/151/256&86/138/249 \\
        &5e5&  77/93/114&68/92/125&112/184/266&110/148/251\\
        \midrule 
        \multirow{3}{*}{\makecell{SDD-based\\level adjustment}}
        &1e5& 12/36/100&32/63/120&16/131/262&66/172/308\\
        &2e5&  97/149/197&128/160/215&207/297/444&238/339/482 \\
        &5e5& 156/205/262&155/205/257&210/343/494&272/409/563 \\
        \midrule 
        \multirow{5}{*}{\makecell{Path-based\\level adjustment}}
        &5e4,1 & 21/30/80&-/-/-&-/-/-&-/-/-  \\
        &5e4,5 &  41/117/232&-/-/-&629/636/-&-/-/- \\
        &5e4,10 & 62/170/439&-/-/-&629/636/-&-/-/-  \\
        &5e4,50 &  279/842/-&-/-/-&159/628/-&-/-/-  \\
        &5e4,100 & 542/-/-&-/-/-&81/-/-&-/-/- \\
        \midrule 
        \multirow{5}{*}{\makecell{Path-based\\level adjustment}}
        &1e5,1 &  13/31/78&-/-/-&-/-/-&-/-/- \\
        &1e5,5 &   37/96/244&-/-/-&629/636/-&-/-/-  \\
        &1e5,10 &  61/172/277&-/-/-&315/320/-&-/-/-  \\
        &1e5,50 &   174/859/-&-/-/-&159/628/-&-/-/-  \\
        &1e5,100 &  328/-/-&-/-/-&43/-/-&-/-/-  \\
        \midrule 
        \multirow{5}{*}{\makecell{Path-based\\level adjustment}}
        &5e5,1 &  72/81/123&-/-/-&-/-/-&-/-/-  \\
        &5e5,5 &  44/88/309&-/-/-&-/-/-&-/-/- \\
        &5e5,10 &  71/165/404&-/-/-&504/509/-&-/-/-  \\
        &5e5,50 &  299/765/-&-/-/-&127/-/-&-/-/-  \\
        &5e5,100 & 537/-/-&-/-/-&65/-/-&-/-/- \\
        \midrule 
        \multirow{5}{*}{\makecell{Path-based\\level adjustment}}
        &1e6,1 &   17/30/73&-/-/-&-/-/-&-/-/-  \\
        &1e6,5 &   44/87/201&-/-/-&504/509/-&-/-/-  \\
        &1e6,10 &  64/167/587&-/-/-&-/-/-&-/-/- \\
        &1e6,50 &   286/767/-&-/-/-&127/-/-&-/-/-  \\
        &1e6,100 &  542/-/-&-/-/-&65/-/-&-/-/-  \\
        \midrule 
        \multirow{7}{*}{\makecell{$s_k=\frac{a}{\sqrt{k}}$}}
        & e-6 &   -/-/-&-/-/-&-/-/-&-/-/-  \\
        & e-5 &   177/200/263&-/-/-&590/637/643&-/-/- \\
        & e-4 &   25/75/139&-/-/-&17/80/-&-/-/- \\
        & e-3 &    510/-/-&-/-/-&437/-/-&-/-/- \\
        & e-2 &   -/-/-&-/-/-&-/-/-&-/-/-  \\
        & e-1 &   -/-/-&-/-/-&-/-/-&-/-/- \\
        & e0 &   -/-/-&-/-/-&-/-/-&-/-/-  \\
        \midrule 
        \multirow{7}{*}{\makecell{$s_k=\frac{a}{k+b}$}}
        & e-6,0 &   -/-/-&-/-/-&-/-/-&-/-/- \\
        & e-5,0 &  -/-/-&-/-/-&-/-/-&-/-/-  \\
        & e-4,0 &  15/35/41&-/-/-&55/66/74&-/-/- \\
        & e-3,0 &   59/91/117&-/-/-&147/202/718&-/-/-  \\
        & e-2,0 &    442/864/-&234/871/-&407/874/-&486/926/-  \\
        & e-1,0 & -/-/-&-/-/-&-/-/-&-/-/- \\
        \midrule 
        \multirow{7}{*}{\makecell{$s_k=\frac{a}{k+b}$}}
        & e-6,10 &   -/-/-&-/-/-&-/-/-&-/-/- \\
        & e-5,10 &   -/-/-&-/-/-&-/-/-&-/-/-  \\
        & e-4,10 &   120/144/223&-/-/-&-/-/-&-/-/- \\
        & e-3,10 &   28/79/112&-/-/-&29/91/665&-/-/- \\
        & e-2,10 &   414/879/-&377/851/-&342/969/-&-/-/-  \\
        & e-1,10 &   -/-/-&-/-/-&-/-/-&-/-/- \\
        \midrule 
        \multirow{7}{*}{\makecell{$s_k=\frac{a}{k+b}$}}
        & e-6,100 &  -/-/-&-/-/-&-/-/-&-/-/- \\
        & e-5,100 & -/-/-&-/-/-&-/-/-&-/-/- \\
        & e-4,100 &  -/-/-&-/-/-&-/-/-&-/-/- \\
        & e-3,100 &  30/33/36&-/-/-&62/68/574&-/-/-  \\
        & e-2,100 &   248/784/-&-/-/-&104/800/-&-/-/-  \\
        & e-1,100 &   -/-/-&-/-/-&-/-/-&-/-/- \\
        \bottomrule 
    \end{tabular}

    }
\end{table}

%% file: comparison_subtable_d801600.tex
\begin{table}[htbp]
    \centering 
    \caption{Results of solving the dual of GAP d801600: comparison between multiple stepsize rules} 
    \label{Comparison_of_rules_d801600}
    {\scriptsize
    \begin{tabular}{cccc}  
        \toprule 
        Method& Param & \multicolumn{2}{c}{d801600}\\
         \cmidrule(lr){3-4} 
         &  &  $x^0: 0$& $x^0: 100$\\
        \midrule 
         \multirow{3}{*}{\makecell{Ours}}
        &1e5&  21/195/358&151/276/433\\
        &2e5& 174/281/446&129/231/395  \\
        &5e5&  198/300/525&145/306/473 \\
        \midrule 
        \multirow{3}{*}{\makecell{SDD-based level adjustment}}
        &1e5& 21/298/491&151/411/604 \\
        &2e5&  387/579/846&425/615/886  \\
        &5e5& 525/793/982&520/790/978\\
        \midrule 
        \multirow{5}{*}{\makecell{Path-based level adjustment}}
        &5e4,1 & -/-/-&-/-/-  \\
        &5e4,5 &  -/-/-&-/-/-\\
        &5e4,10 & -/-/-&-/-/- \\
        &5e4,50 &  332/337/-&-/-/- \\
        &5e4,100 & 167/-/-&-/-/-\\
        \midrule 
        \multirow{5}{*}{\makecell{Path-based level adjustment}}
        &1e5,1 &  -/-/-&-/-/- \\
        &1e5,5 &   -/-/-&-/-/- \\
        &1e5,10 &  -/-/-&-/-/-  \\
        &1e5,50 &   332/337/-&-/-/-  \\
        &1e5,100 &  167/-/-&-/-/-  \\
        \midrule 
        \multirow{5}{*}{\makecell{Path-based level adjustment}}
        &5e5,1 &  -/-/-&-/-/- \\
        &5e5,5 &  -/-/-&-/-/- \\
        &5e5,10 &  -/-/-&-/-/- \\
        &5e5,50 &  266/-/-&-/-/- \\
        &5e5,100 & 135/-/-&-/-/-\\
        \midrule 
        \multirow{5}{*}{\makecell{Path-based level adjustment}}
        &1e6,1 &   -/-/-&-/-/-  \\
        &1e6,5 &   -/-/-&-/-/-  \\
        &1e6,10 &  -/-/-&-/-/-\\
        &1e6,50 &   266/-/-&-/-/- \\
        &1e6,100 &  135/-/-&-/-/- \\
        \midrule 
        \multirow{7}{*}{\makecell{$s_k=\frac{a}{\sqrt{k}}$}}
        & e-6 &   -/-/-&-/-/-    \\
        & e-5 &  -/-/-&-/-/- \\
        & e-4 &  30/282/-&-/-/-\\
        & e-3 &   933/-/-&-/-/- \\
        & e-2 &  -/-/-&-/-/- \\
        & e-1 &  -/-/-&-/-/-\\
        & e0 &   -/-/-&-/-/- \\
        \midrule 
        \multirow{7}{*}{\makecell{$s_k=\frac{a}{k+b}$}}
        & e-6,0 & -/-/-&-/-/-\\
        & e-5,0 & -/-/-&-/-/-  \\
        & e-4,0 &  -/-/-&-/-/- \\
        & e-3,0 &  -/-/-&-/-/- \\
        & e-2,0 &  -/-/-&-/-/- \\
        & e-1,0 & -/-/-&-/-/-\\
        \midrule 
        \multirow{7}{*}{\makecell{$s_k=\frac{a}{k+b}$}}
        & e-6,10 &  -/-/-&-/-/-\\
        & e-5,10 &  -/-/-&-/-/- \\
        & e-4,10 &  -/-/-&-/-/-\\
        & e-3,10 &  25/168/-&-/-/-\\
        & e-2,10 & -/-/-&-/-/- \\
        & e-1,10 &  -/-/-&-/-/-\\
        \midrule 
        \multirow{7}{*}{\makecell{$s_k=\frac{a}{k+b}$}}
        & e-6,100 & -/-/-&-/-/- \\
        & e-5,100 & -/-/-&-/-/- \\
        & e-4,100 &  -/-/-&-/-/-\\
        & e-3,100 &  150/165/-&-/-/- \\
        & e-2,100 &  177/-/-&-/-/- \\
        & e-1,100 &   -/-/-&-/-/-\\
        \bottomrule 
    \end{tabular}

    }
\end{table}